\documentclass[11pt,reqno]{amsart}
\usepackage[margin=1in,letterpaper]{geometry}
\usepackage{graphicx}
\usepackage{amssymb}
\usepackage{amsthm}
\usepackage{mathrsfs}
\usepackage{stmaryrd}
\usepackage{accents}
\usepackage{enumitem}
\usepackage[mathscr]{euscript}
\usepackage{latexsym}
\usepackage{hyperref}
\hypersetup{pdfstartview={XYZ null null 1.00}} 
\usepackage{bbm}
\usepackage{mathrsfs}
\usepackage{amsfonts,amsmath,amsthm,amssymb,stmaryrd}
\newtheorem{theorem}{Theorem}[section]

\newtheorem{lemma}[theorem]{Lemma}
\newtheorem{Proposition}[theorem]{Proposition}
\newtheorem{Remark}{Remark}

\numberwithin{equation}{section} \allowdisplaybreaks

\def\r3{\mathbb{R}^3}

\begin{document}
\title[Optimal large time behavior of the compressible BNSP system] {Optimal large time behavior of the
compressible Bipolar Navier--Stokes--Poisson system with unequal
viscosities}
\author{Qing Chen}
\address{Qing Chen \newline School of Applied Mathematics\\
Xiamen University of Technology\\
Xiamen, Fujian 361024, China} \email{chenqing@xmut.edu.cn}

\author{Guochun Wu}
\address{Guochun Wu \newline Fujian Province University Key Laboratory of Computational
Science,School of Mathematical Sciences, Huaqiao University,
Quanzhou 362021, China} \email{guochunwu@126.com}

\author{Yinghui Zhang*}
\address{Yinghui Zhang \newline School of Mathematics and Statistics, Guangxi Normal
University, Guilin, Guangxi 541004, China}
\email{yinghuizhang@mailbox.gxnu.edu.cn}

\thanks{* Corresponding author:
yinghuizhang@mailbox.gxnu.edu.cn} 

\begin{abstract}
 This paper is concerned with the Cauchy problem of the 3D compressible bipolar
 Navier--Stokes--Poisson (BNSP) system with unequal viscosities, and our main purpose
 is three--fold: First, under the assumption that $H^l\cap L^1$($l\geq 3$)--norm of
 the initial data is small, we prove the optimal time decay rates of the solution as well as
 its all--order spatial derivatives from one--order to the highest--order, which are the same
 as those of the compressible Navier--Stokes equations and the heat equation. Second, for
 well--chosen initial data, we also show the lower bounds on the decay rates. Therefore,
 our time decay rates are optimal. Third, we give the explicit influences of the electric field
 on the qualitative behaviors of solutions, which are totally new as compared to the results for the compressible unipolar Navier--Stokes--Poisson(UNSP) system [Li et al., in Arch. Ration. Mech. Anal., {196} (2010), 681--713; Wang, J. Differ. Equ., { 253} (2012), 273--297]. More precisely, we show that the densities of the BNSP system converge to their corresponding equilibriums at the same $L^2$--rate $(1+t)^{-\frac{3}{4}}$ as the compressible Navier--Stokes equations, but the momentums of the BNSP system and the difference between two densities decay at the $L^2$--rate $(1+t)^{-\frac{3}{2}(\frac{1}{p}-\frac{1}{2})}$ and $(1+t)^{-\frac{3}{2}(\frac{1}{p}-\frac{1}{2})-\frac{1}{2}}$ with $1\leq p\leq \frac{3}{2}$, respectively, which depend directly on the initial low frequency assumption of electric field, namely, the smallness of $\|\nabla \phi_0\|_{L^p}$.
\end{abstract}

\maketitle


\section{Introduction}

At high temperature and velocity, ions and electrons in a plasma tend to become two separate fluids due to their different physical properties (inertia, charge). One of the fundamental fluid models for describing plasma dynamics is the two--fluid model, in which two compressible ion and electron fluids penetrate each other through their own self--consistent electromagnetic field. In this paper, we formally take the magnetic field equal to zero and study the optimal time decay rate for the global classical solution to the bipolar Navier--Stokes--Poisson (BNSP) system in 3D:
\begin{equation}\label{1eq}
 \left\{\begin{array}{lll}
  \partial_t \rho_1 + {\rm div} m_1 =0,
  \\
  \partial_t m_1 + {\rm div}\left(\frac{m_1\otimes m_1}{\rho_1}\right)+\nabla P_1 =
  \mu_1 \Delta \left(\frac{m_1}{\rho_1}\right) +\nu_1\nabla {\rm div}\left(\frac{m_1}{\rho_1}\right) + Z\rho_1\nabla \phi,
  \\
  \partial_t \rho_2 + {\rm div} m_2=0,
  \\
  \partial_t m_2 + {\rm div}\left(\frac{m_2\otimes m_2}{\rho_2}\right)+\nabla P_2 =
  \mu_2 \Delta \left(\frac{m_2}{\rho_2}\right) + \nu_2\nabla {\rm div}\left(\frac{m_2}{\rho_2}\right) -\rho_2\nabla \phi ,
  \\
  \Delta \phi =Z\rho_1 - \rho_2, \quad \lim\limits_{|x|\rightarrow\infty} \phi(x, t) =0
 \end{array}\right.
\end{equation}
for $(x, t)\in \mathbb{R}^3 \times \mathbb{R}^+$ and with initial data
\begin{equation}\label{1id}
 (\rho_1, m_1, \rho_2, m_2, \phi)(x,0) = (\rho_{10}, m_{10}, \rho_{20}, m_{20}, \phi_0)(x) \rightarrow \left(\frac1Z, 0, 1, 0, 0\right) \quad as\quad |x|\rightarrow \infty.
\end{equation}
The (BNSP) system describes a plasma composed of ions and electrons.
The unknown functions are density $\rho_1$ and momentum $m_1$ of the
ions, density $\rho_2$ and momentum $m_2$ of the electrons, and the
electrostatic potential $\phi$. The positive constant $Z$ represents
the charge of the ions. For $i=1,2$, the viscosity coefficients
$\mu_i>0$ and $\nu_i$ satisfy $3\mu_i +2 \nu_i>0$, and the pressure
functions $P_i$ of the two fluids terms of $\rho_i$ are smooth
functions and satisfy $P_i'(\rho_i)>0$ if $\rho_i>0$. Without loss
of generality, we assume other physical parameters to be $1$.

Owing to the physical importance and mathematical challenges, there
is an extensive literature on the long time behavior of global
smooth solutions to the Navier--Stokes--Poisson (NSP) system. If
only considering the dynamics of one fluid in plasmas, then the
system \eqref{1eq} reduces to the unipolar Navier--Stokes--Poisson
(UNSP) system. For the UNSP system, Li--Matsumura--Zhang \cite{LMZ}
proved that the density of the NSP system converges to its
equilibrium state at the same $L^2$--rate $(1+t)^{-\frac34}$ as the
compressible Navier--Stokes (NS) system, but the momentum of the
UNSP system decays at the $L^2$--rate $(1+t)^{-\frac14}$, which is
slower than the $L^2$--rate $(1+t)^{-\frac34}$ for the NS system.
And then they extended similar result to the non--isentropic case
\cite{ZLZ}. Wu--Wang \cite{WW1} investigated the pointwise estimates
of the solution and showed the pointwise profile of the solution
contains the D--wave but does not contain the H--wave, which is
different from the Navier--Stokes system. Wang \cite{W1} obtained
the optimal asymptotic decay of solutions just by pure energy
estimates. Hao--Li \cite{HL}, Tan--Wu \cite{TW}, Chikami--Danchin
\cite{CD} and Bie--Wang--Yao \cite{BWY} also established the unique
global solvability and the optimal decay rates for small
perturbations of a linearly stable constant state. We mention that
there are many results on the existence and long time behavior of
the weak solutions or non--constant stationary solutions, see, for
example \cite{Bella1, Donatelli1, Ducomet1, Tan3, ZhangY} and the
references therein. We also mention that the quasi--neutral
phenomenon to the UNSP system are studied in \cite{Donatelli2, Ju,
Wangsu}.

\par
For the BNSP system \eqref{1eq},  there are very few results due to its non--conservative structure and the interaction of two fluids through the electric field. Duan--Yang \cite{DY} proved global well--posedness and asymptotic behavior of smooth solutions for the Cauchy problem in one
dimension, and Zhou--Li \cite{ZL} obtained the corresponding convergence rate. Recently,
Li--Yang--Zou \cite{LYZ}, Zou \cite{Z1}, and Wu--Wang \cite{WW1} obtained the
global existence and the optimal decay rates of the classical solution around a constant state
by a detailed analysis of the Greens function to the corresponding linearized equations, and Wang--Xu \cite{WX}
also proved the $L^2$--decay rate of the solution by using long wave and short wave decomposition method. It should be noted that in \cite{LYZ, WW1, Z1}, the viscosity coefficients of two fluids are
taken to be equal to each other, i.e.,
$$\mu_1 = \mu_2, \quad \nu_1 =\nu_2.$$
Thus by taking a linear combination of the system \eqref{1eq},
the system \eqref{1eq} can be reformulated into one for the Navier--Stokes system and another one for the UNSP system which however are coupled with each other through the nonlinear terms. Thus in order to obtain a priori estimates of the solutions, one can apply the similar arguments as in \cite{C1, Mats1, Mats2} for the Navier--Stokes system and in \cite{LMZ, W1} for the UNSP system. Recently, under the assumption that the initial perturbation is small in $H^l(\mathbb{R}^3)$ with $l\geq3$, Wu--Zhang--Zhang \cite{WZZ} established the global existence for the BNSP system with unequal viscosities, i.e.,
\begin{equation}\label{une-vis}
 \mu_1 \neq \mu_2, \quad \nu_1 \neq \nu_2.
\end{equation}
Moreover, if in addition, the initial perturbation is small in $\dot{H}^{-s}(\mathbb{R}^3)$--norm with $\frac12\le s<\frac32$ or $\dot B^{-s}_{2, \infty}$--norm with $\frac12<s\leq\frac32$, it is shown that the densities $\rho_i$ and the velocities $u_i =\frac{m_i}{\rho_i}$ with $i = 1,2$ have the following convergent decay estimates
\begin{equation}\nonumber
 \displaystyle
 \left\|\nabla^k\left(\rho_{1}-\frac{1}{Z},u_{1}, \rho_{2}-1,u_{2},\nabla\phi\right)(t) \right\|_{\ell-k}\lesssim(1+t)^{-\frac{k+s}{2}},
\end{equation}
for $k=0,1,\cdots, l-1,$ which together with the fact that for $p\in (1, 3/2]$, $L^p\subset \dot H^{-s}$ with $s=3(\frac{1}{p}-\frac{1}{2})\in [1/2, 3/2)$, and for $p\in [1, 3/2)$, $L^p \subset \dot B^{-s}_{2, \infty}$ with $s=3(\frac{1}{p}-\frac{1}{2})\in (1/2, 3/2]$ imply
\begin{equation}\label{1.4}
 \displaystyle
 \left\|\nabla^k\left(\rho_{1}-\frac{1}{Z}, u_{1}, \rho_{2}-1, u_{2}, \nabla\phi\right)(t)\right\|_{\ell-k}
 \lesssim (1+t)^{-{\frac{3}{2}(\frac{1}{p}-\frac{1}{2})-\frac{k}{2}}},
\end{equation}
for $p\in [1, 3/2]$ and $k=0,1,\cdots, l-1$. For the compressible Euler--Poisson system and related models, we refer to \cite{Guo2, Guo3, Guo4, WuQin, WuW11} and the references therein.

\par
\bigskip
Noticing the decay rates in \eqref{1.4}, for the $l$--order (i.e. \textbf{the highest--order}) spatial derivative  of the solution, it
holds that
\begin{equation}\label{1.5}
 \left\|\nabla^l\left(\rho_{1}-\frac{1}{Z}, u_{1}, \rho_{2}-1, u_{2}, \nabla\phi \right)(t)\right\|_{L^{2}}
 \lesssim(1+t)^{-{\frac{3}{2}(\frac{1}{p}-\frac{1}{2})-\frac{l-1}{2}}}.
\end{equation}
On the other hand, let us revisit the following classical result of the heat equation:
\begin{equation}\label{1.6}
 \left\{\begin{array}{l}\partial_{t} u-\Delta u=0, ~\text{ in } ~\mathbb{R}^{3},
 \\
 u|_{t=0}=u_0.
 \end{array}\right.
\end{equation}
If $u_0\in H^l(\mathbb{R}^{3})\cap L^p(\mathbb{R}^{3})$ with $l\geq 0$ be an integer and $1\leq p\leq 2$, then for any $0\leq k\leq l$, the solution of the heat equation \eqref{1.6} has the following decay rate:
\begin{equation}\nonumber
 \|\nabla^k u(t)\|_{L^2}
 \leq C((1+t)^{-{\frac{3}{2}(\frac{1}{p}-\frac{1}{2})-\frac{k}{2}}},
\end{equation}
which particularly implies
\begin{equation}\label{1.8}
 \|\nabla^l u(t)\|_{L^2}
 \leq C(1+t)^{-{\frac{3}{2}(\frac{1}{p}-\frac{1}{2})-\frac{l}{2}}}.
\end{equation}
Therefore, in view of \eqref{1.5} and \eqref{1.8}, it is clear that
the decay rate of the $l$--order spatial derivative of the solution
in \eqref{1.5} is slower than that of the heat equation as in
\eqref{1.8}. So, the decay rate of the $l$--order spatial derivative
of the solution in \eqref{1.5} is not optimal in this sense.
Furthermore, the decay rates in \eqref{1.4} give no information on
influences of the electric field on the qualitative behaviors of
solutions.\par The main motivation of this article is to give a
clear answer to these issues mentioned above. More precisely, our
main new contributions can be outlined as follows: First, our
methods provide a general framework that can be used to extract the
optimal decay rates of the solution to the Cauchy problem
\eqref{1eq}--\eqref{1id} as well as its all--order spatial
derivatives from one--order to the highest--order, which are the
same as those of the heat equation. Second, for well--chosen initial
data, we also show the lower bounds on the decay rates. Therefore,
our time decay rates are optimal. Third, we give the explicit
influences of the electric field on the qualitative behaviors of
solutions, which are completely new as compared to the previous
results for the UNSP system [Li et al., in Arch. Ration. Mech.
Anal., {196} (2010), 681--713; Wang, J. Differ. Equ., { 253} (2012),
273--297], and our results also imply the same effects of
electrostatic potential on the UNSP system immediately.

Before stating our results, let us introduce some notations and conventions used throughout this paper. We employ $H^\ell(\mathbb R^3)$ and $W^{m,p}(\mathbb R^3)$ to denote the usual Sobolev spaces with norm $||\cdot||_m$ and $||\cdot||_{m,p}$ for $m\geq 0$ and $1\leq p \leq+\infty$. If $m=0$, we just use $||\cdot||_{L^2}$ and $||\cdot||_{L^q}$ for convenience. Set a radial function $\varphi\in C_0^\infty(\mathbb R^3_\xi)$ such that $\varphi(\xi)=1$ while $|\xi| \le 1$ and $\varphi(\xi)=0$ while $|\xi| \geq 2$. Define the low frequent part of $f$ by
\begin{equation}\nonumber
 f^L= \mathfrak{F}^{-1}[\varphi(\xi)\widehat{f}],
\end{equation}
and the high frequent part of $f$ by
\begin{equation}\nonumber
 f^H= \mathfrak{F}^{-1}[(1-\varphi(\xi))\widehat{f}],
\end{equation}
then $f= f^L +f^H$ if the Fourier transform of $f$ exists. We will employ the notation $a\lesssim b$ to mean that $a\leq Cb$ for a universal constant $C>0$ that only depends on the parameters coming from the problem. And $C_i(i=1,2,\cdots, 9)$ will also denote some positive constants depending only on the parameters of the problem.

Our main results are stated in the following theorem.
\begin{theorem}\label{1mainth} $\bullet$ \textbf{Global existence.}
 Assume that $\left(\rho_{10}-\frac{1}{Z}, m_{10}, \rho_{20}-1, m_{20}, \nabla \phi_0\right) \in H^l(\mathbb{R}^3)$ with $l\geq 3$. Then if there exists a sufficiently small constant $\delta_0>0$, such that
 \begin{equation}\label{id-delta0}
  C_0
  = \left\|\left(\rho_{10}-\frac{1}{Z}, m_{10}, \rho_{20}-1, m_{20}, \nabla \phi_0\right)\right\|_{l} \le \delta_0,
 \end{equation}
 then the Cauchy problem \eqref{1eq}--\eqref{1id} with unequal viscosities \eqref{une-vis} admits a unique globally classical solution $(\rho_1, m_{1}, \rho_2, m_{2}, \phi)$ satisfying that for any $t\in [0, \infty)$,
 \begin{equation}\nonumber
  \begin{split}
  &\left\|\left(\rho_1-\frac{1}{Z}, m_1, \rho_2 -1, m_2, \nabla \phi\right)\right\|_{l}^2
  +\int_0^t \left(\left\|\nabla\left(\rho_1, \rho_2, \nabla \phi\right)(\tau)\right\|_{l-1}^2 +
  \|\nabla(m_1, m_2)(\tau)\|_{l}^2\right) d\tau
  \\
  &\le C\left\|\left(\rho_{10}-\frac{1}{Z}, m_{10}, \rho_{20}-1, m_{20}, \nabla
  \phi_0\right)\right\|_{l}^2.
  \end{split}
 \end{equation}
 $\bullet$ \textbf{Upper decay rates.}
 Under the assumption in Theorem \ref{1mainth}, if additionally for $1\le p\le \frac32$,
 \begin{equation}\label{k0}
  K_0
  =\left\|\left(\rho_{10}-\frac1Z, m_{10}, \rho_{20}-1, m_{20}\right)\right\|_{L^1} + \| \nabla\phi_0\|_{L^p}<\delta_0,
 \end{equation}
 then for all $t\geq 0$, it holds that
 \begin{equation}\label{1.13}
  \left\|\nabla^k\left(\rho_{1}-\frac{1}{Z}, \rho_{2}-1\right)\right\|_{L^2}
  \le C(1+t)^{-\frac34 -\frac k2},
 \end{equation}
 and
 \begin{equation}\label{1.14}
  \|\nabla^k(m_{1}, m_{2}, \nabla\phi)\|_{L^2}
  \le C(1+t)^{-\frac32 \left(\frac 1p-\frac12\right) -\frac k2},
 \end{equation}
 for $0\le k\le l$.

 \noindent$\bullet$ \textbf{Lower decay rates.} Moreover, assume that \eqref{k0} holds for $p=1$ and the Fourier transform functions $(\rho_{10}-1, m_{10}, \rho_{20}-1, m_{20})$ satisfy
 \begin{equation}\label{low-as}
  \mathfrak{F}\left[\rho_{10}-\frac1Z\right]
  =\mathfrak{F}[m_{10}]
  =\mathfrak{F}[\rho_{20}-1]
  = \mathfrak{F}[\Lambda^{-1}{\rm curl} m_{20}]
  =0,
  \quad \hbox{and} \quad |\mathfrak{F}[\Lambda^{-1}{\rm div} m_{20}]| \geq C\delta_0^\frac32,
 \end{equation}
 for any $|\xi|\le \eta$. Then there exists a positive constant $C_1$ independent of time such that for any large enough $t$,
 \begin{equation}\nonumber
  \begin{split}
  &\min\left\{\left\|\nabla^k\left(\rho_{1}-\frac1Z\right)\right\|_{L^2}, \left\|\nabla^km_{1}\right\|_{L^2}, \left\|\nabla^k\left(\rho_{2}-1\right)\right\|_{L^2}, \left\|\nabla^km_{2}\right\|_{L^2}, \left\|\nabla^k\nabla\phi\right\|_{L^2}\right\}
  \\
  &\geq C_1\delta_0^\frac32(1+t)^{-\frac34 -\frac k2},
  \end{split}
 \end{equation}
 for $0\le k\le l$.
\end{theorem}

\begin{Remark}
 The global existence of the Cauchy problem \eqref{1eq}--\eqref{1id} with unequal viscosities \eqref{une-vis} under
 the small initial perturbation assumption has been proved in \cite{WZZ} by the standard continuity argument. In this paper,
 we focus on the upper--lower time decay rates of the solution as well as all-orders spatial derivatives, and
 the explicit influences of the electric field on the qualitative behaviors of solutions.
\end{Remark}

\begin{Remark}
 It is interesting to make a comparison between Theorem \ref{1mainth} and those of the UNSP system \cite{LMZ, W1}. On the one hand, the main result of \cite{LMZ} can be listed as follows:
 \par
 Let $l\geq 4$ be an integer, if $||(\rho_0-\bar{\rho}, m_0)||_{H^l\cap L^1}$ is sufficiently small, the authors in \cite{LMZ} proved that the UNSP system has a small smooth solutions satisfying the following decay rates with $k=0,1$:
 \begin{equation}\label{1.17}
  \|\nabla^k(\rho-\bar{\rho})\|_{L^2}
  \leq C (1+t)^{-\frac34-\frac k2}\|(\rho_0-\bar{\rho}, m_0)\|_{H^l\cap L^1},
 \end{equation}
 \begin{equation}\label{1.18}
  \|\nabla^k(m, \nabla \phi)\|_{L^2}\leq C (1+t)^{-\frac14-\frac k2}||(\rho_0-\bar{\rho}, m_0)||_{H^l\cap L^1}.
 \end{equation}
 It should be mentioned that there is no initial low frequency assumption on the electric field $\phi$ in their proofs,
 which is the reason why the decay rate of the momentum of the UNSP system \eqref{1.18} is slower than that of the
 compressible Navier--Stokes equations. In addition, the decay rates \eqref{1.17} and \eqref{1.18}
 give no information on the optimal decay rates of the higher--order spatial derivatives ($k\geq 2$) of the solutions.
 In our Theorem \ref{1mainth}, by replacing \eqref{k0} with
 \begin{equation}\label{1.19}
  \left\|\left(\rho_{10}-\frac1Z, m_{10}, \rho_{20}-1, m_{20}\right)\right\|_{L^1}<\delta_0,
 \end{equation}
 we can also show that \eqref{1.13} and the following decay rate hold,
 \begin{equation}\label{1.20}
  \|\nabla^k(m_{1}, m_{2}, \nabla\phi)\|_{L^2} \le C(1+t)^{-\frac14-\frac k2},~\hbox{for}~0\le k\le l,
 \end{equation}
 where the decay rates are the same as those of \eqref{1.17} and \eqref{1.18}, and also give the optimal decay rates of the higher--order spatial derivatives ($k\geq 2$) of the solutions. Moreover, this particularly implies that the decay rate \eqref{1.14} in Theorem \ref{1mainth} can be replaced by
 \begin{equation}\label{1.21}
  \|\nabla^k(m_{1}, m_{2}, \nabla\phi)\|_{L^2}
  \leq
  \left\{\begin{array}{lll}
  C(1+t)^{-\frac32 \left(\frac 1p-\frac12\right) -\frac k2},~~\hbox{for}~~1\leq p<\frac32,
  \\
  C(1+t)^{-\frac14-\frac k2},~~\hspace{1cm}\hbox{for}~\frac32\leq p\leq 2,
  \end{array}\right.
 \end{equation}
 for $0\le k\le l$. To see this, under the assumption \eqref{1.19}, by Hausdorff--Young inequality, one can modify the proofs of \eqref{li-v-1} in Proposition \ref{li-de-es1} to obtain
 \begin{equation}\nonumber
  \|\nabla^k( n_1 ^L, n_2^L, ( \nabla\phi)^L )\|_{L^2}
  \le C (1+t)^{-\frac14 -\frac k2}\|U_0^L\|_{L^1},
 \end{equation}
 which together with the nonlinear energy estimates used in the proof of Proposition \ref{es-thm-M} yields \eqref{1.13} and \eqref{1.20}
 immediately. On the other hand, the main result of \cite{W1} can be outlined as follows:
 \par
 Assume that $\rho_{0 }-1, u_{0 },\nabla\phi_0\in H^{l}\cap \dot{H}^{-s}$ for an integer $l\ge 3$ and $s\in [0,3/2)$.
 Then there exists a sufficiently small constant $\epsilon_0$ such that if
 \begin{equation}\nonumber
  \|{(\rho_{0 }-1, u_0, \nabla\phi_0)}\|_{{3}}\leq \epsilon_0,
 \end{equation}
 then the UNSP system admits a small smooth solutions satisfying the following decay rates $(\rho,u,\nabla\phi)$ satisfying the following decay rates:
 \begin{equation}\nonumber
  \|\nabla^k(\rho-1, u, \nabla\phi)\|_{l-k}\leq C\|(\rho_{0 }-1, u_{0}, \nabla\phi_0)\|_{H^{l}\cap \dot{H}^{-s}}(1+t)^{-\frac{k+s}{2}}\
  \hbox{ for }k=0,\dots, l-1,
 \end{equation}
 \vspace{-0.5cm}
 \begin{equation}\nonumber
  \|{\nabla^k(\rho-1)(t)}\|_{l-k}\leq C\|(\rho_{0 }-1, u_{0}, \nabla\phi_0)\|_{H^{l}\cap\dot{H}^{-s}}(1+t)^{-\frac{k+s+1}{2}}\ \hbox{ for }k=0,\dots, l-2,
 \end{equation}
 which together with the fact that for $p\in (1,2]$, $L^p\subset \Dot{H}^{-s}$ with $s=3(\frac{1}{p}-\frac{1}{2})\in[0,3/2)$ imply that the following optimal decay results:
 \begin{equation}\nonumber
  \|\nabla^k{(\rho-1, u, \nabla\phi)}\|_{L^2}
  \leq C\|(\rho_{0 }-1, u_{0}, \nabla\phi_0)\|_{H^{l}\cap L^p} (1+t)^{-\frac32 \left(\frac 1p-\frac12\right)-\frac k2}\
  \hbox{ for }k=0,\dots, l-1,
 \end{equation}
 \vspace{-0.5cm}
 \begin{equation}\label{1decay2}
  \|{\nabla^k(\rho-1)(t)}\|_{L^2}
  \leq C\|(\rho_{0 }-1, u_{0}, \nabla\phi_0)\|_{H^{l}\cap L^p} (1+t)^{-\frac32 \left(\frac 1p-\frac12\right)-\frac{k+1}2}\
  \hbox{ for }k=0,\dots, l-2.
 \end{equation}
 Particularly, the decay rate in \eqref{1decay2} implies that the initial low frequency assumption on the electric field $\phi$ enhances the decay rate of the density $\rho$.
 \par
 In conclusion, both of the results in \cite{LMZ, W1} imply that the electric field affects
 the decay rates of the solutions. However, neither of them clarifies that how the electric field affects the decay rates of the solutions.
 Our Theorem \ref{1mainth} gives a clear answer to this issue. More precisely,
 the decay rate in \eqref{1.21} shows the explicit influences of the electric field on the decay rates of solutions,
 which is totally new as compared to the previous results for the UNSP system \cite{LMZ, W1}.
 On the other hand, by employing our method to the UNSP system, one can easily find that the electric field plays the same
 role in doing decay rate to the two systems.
Therefore, this phenomenon is the most important difference between
the NS system and the NSP system. In addition, by noticing the fact
that $Z\rho_1-\rho_2\sim \Delta\phi$ and using \eqref{1.21}, the
difference $Z\rho_1-\rho_2$ between two densities has the following
decay rate
 \begin{equation}\nonumber
  \|\nabla^k (Z\rho_1-\rho_2)\|_{L^2}
  \leq
  \left\{\begin{array}{lll}
   C(1+t)^{-\frac32 \left(\frac 1p-\frac12\right) -\frac {k+1}2},~~\hbox{for}~~1
   \leq p<\frac32,
   \\
   C(1+t)^{-\frac14-\frac {k+1}2},~~\hspace{1cm}\hbox{for}~\frac32
   \leq p\leq 2,
  \end{array}\right.
 \end{equation}
 for $0\le k\le l-1$, which is faster than those of themselves. Finally, for well--chosen initial data, we also get the lower bound on the decay rates. Therefore, the time decay rates in Theorem \ref{1mainth} are optimal.
\end{Remark}

\indent Now, let us sketch the strategy of proving Theorem
\ref{1mainth} and explain some of the main difficulties and
techniques involved in the process. As mentioned before, the main
ideas of \cite{LYZ,WW1,Z1} are based on reformulating the system
\eqref{1eq} by taking linear combination of the system \eqref{1eq}
and employing the similar arguments as in \cite{C1, Mats1, Mats2}
for the Navier--Stokes system and in \cite{LMZ, W1} for the UNSP
system. Therefore, the methods in \cite{LYZ,WW1,Z1}, depending
essentially on this reformulation, do not work here. The main idea
here is that instead of using the reformulation, we will work on the
system \eqref{1eq} directly. So, compared to \cite{LYZ,WW1,Z1}, we
need to develop new ingredients in the proof to overcome the
difficulties arising from unequal viscosities, the non--conservative
structure of the system \eqref{1eq} and the interaction of two
fluids through the electric field, which requires some new ideas.
More precisely, we will employ ``div--curl" decomposition, the
low--frequency and high--frequency decomposition, delicate spectral
analysis and energy estimates. Roughly speaking, our proof mainly
involves the following five steps.\par First, we rewrite the Cauchy
problem \eqref{1eq}--\eqref{1id} into the  perturbation form
\eqref{2eq-m} and analyze the spectral of the solution semigroup to
the corresponding linear system. We therefore encounter a
fundamental obstacle that the Fourier transform of the Green's
matrix to the linear system of \eqref{2eq-m} is an $8$--order matrix
and is not self--adjoint. Particularly, it is easy to check that it
can not be diagonalizable (see \cite{Sideris} pp.807 for example).
Therefore, it is very difficult to apply the usual time decay
investigation through spectral analysis. To tackle with this
problem, we will employ the ``div--curl" decomposition technique
developed in \cite{Dan1, Dan2, Dan3, WZZ2} to split the linear
system into three systems. One has four distinct eigenvalues and the
other two are classic heat equations. Then, by making careful
analysis on the Fourier transform of Green's function to the linear
equations, we can obtain the desired linear decay rates. More
exactly, from the elaborate expression \eqref{so-max1} of the
solution semigroup, we find that the low frequency of the
electrostatic potential plays crucial role in the decay rate
estimate. Although the solution semigroup of the BNSP system is much
more complicated than one of the UNSP system, explicit spectral
analysis on the semigroup means that the influence of the
electrostatic potential to these two systems is almost the same.
Meanwhile, it is also the most important difference between the NS
system and the NSP system (see Proposition \ref{li-de-es1} and
\ref{prop-decay4} for details).
\par
Second, we deduce zero--order and the highest--order low frequency
decay estimates. In the process of deducing the zero--order low
frequency decay estimates, the main difficulty lies in deriving the
decay rates of the densities, which are faster than those of the
momentums. Indeed, noting that the expression of the solution
\eqref{so-ex}, \eqref{li-varrho-1} and the estimate of the nonlinear
term \eqref{es-no-N}, one can get
 \begin{equation}
 \begin{split}\label{1.19}
 \|(\varrho_1^L, \varrho_2^L)(t)\|_{L^2}
 &
\lesssim
(1+t)^{-\frac34}\|U_0\|_{L^1}+\underbrace{\int_0^{t}\|((S^\varrho_1)^L,
 (S^\varrho_2)^L)\|_{L^2}d\tau}_{\mathcal{I}(t)}\\
 &\lesssim (1+t)^{-\frac34}\|U_0\|_{L^1} + \int_0^{t}(1+t-\tau)^{-\frac34} \|\mathcal{N}^L(\tau)\|_{L^1}d\tau
 \\
 &\lesssim (1+t)^{-\frac34}\|U_0\|_{L^1}+ \mathcal{M}^2(t)\int_0^{t}(1+t -\tau)^{-\frac34}(1+\tau)^{-\frac32\left(\frac1p -
 \frac12\right)-\frac34}d\tau,
 \end{split}
\end{equation}
where $\mathcal{M}(t)$ is defined in \eqref{M1}. However, by taking
$p=\frac32$, it is clear that
\begin{equation}\label{1.20}
 \int_0^{t}(1+t -\tau)^{-\frac34}(1+\tau)^{-\frac32\left(\frac1p -
 \frac12\right)-\frac34}d\tau=\int_0^{t}(1+t-\tau)^{-\frac34}(1+\tau)^{-1}d\tau
 =O((1+t)^{-\frac34}\hbox{ln}(1+t)).
\end{equation}
In view of \eqref{1.19} and \eqref{1.20}, it seems impossible to
obtain the decay estimate of $(\varrho_1^L, \varrho_2^L)$ as in
\eqref{low-fre-varrho1} which however is crucial for the proof of
Theorem \ref{1mainth}. Our key idea here is to split
$\mathcal{I}(t)$ into two parts:
\begin{equation}\nonumber
 \begin{split}
 \mathcal{I}(t)
 &=\left(\int_0^{\frac t2}+\int_{\frac t2}^t\right)\|((S^\varrho_1)^L,
 (S^\varrho_2)^L)\|_{L^2}d\tau
 \\
 &:=\mathcal{I}_1(t)+\mathcal{I}_2(t),
 \end{split}
\end{equation}
and then estimate the terms $\mathcal{I}_1(t)$ and $\mathcal{I}_2(t)$ respectively. For the term $\mathcal{I}_2(t)$, one can easily get
\begin{equation}\nonumber
 \mathcal{I}_2(t)
 \lesssim \mathcal{M}^2(t)\int_{\frac t2}^{t}(1+t-\tau)^{-\frac34}(1+\tau)^{-\frac32\left(\frac1p -
 \frac12\right)-\frac34}d\tau\lesssim \mathcal{M}^2(t)(1+t)^{-\frac34},
\end{equation}
for any $p\in [1, \frac32]$. However, we need develop new thoughts
to deal with the term $\mathcal{I}_1(t)$. More precisely, we have to
figure out how to improve the estimates on the terms ${\rm
div}\left(\frac{m_i\otimes m_i}{\rho_i}\right)$ in \eqref{N1N2-m},
which devote the slowest decay rate in $\mathcal{I}_1(t)$. To this
end, we try our best to explore the nonlinear part in the expression
of the solution \eqref{so-ex}. Fortunately, by delicate
calculations, we surprisingly find that the nonlinear terms in the
expressions of the Fourier transformations of the densities
$(\varrho_1^L, \varrho_2^L)$ can be rewritten in divergent forms(See
the proof of Proposition \ref{prop-decay4} for details). Note that
this is one of the differences between the ``densities--momentums"
system \eqref{2eq-m} and the ``densities--velocities" system
\eqref{2eq}. As a result, one can shift the derivative onto the
solution semigroup to obtain the desired decay estimates. With the
help of this key observation, one can obtain the same decay rate of
the term $\mathcal{I}_1(t)$ as that of $\mathcal{I}_2(t)$.
Consequently, we can get the decay estimate of $(\varrho_1^L,
\varrho_2^L)$ in \eqref{low-fre-varrho1}. To obtain the
highest--order low frequency decay estimates, as compared to the
case of zero-order, we encounter a new difficulty that it requires
us to control the terms involving $l+1$--order or $l+2$--order
spatial derivatives of the solutions which however don't belong to
the solution space. To get around this difficulty, we separate the
time interval into two parts and make full use of the benefit of the
low--frequency and high--frequency decomposition to get our desired
convergence rates (see the proof of Lemma \ref{le-es-loworder1} for
details).
\par
Third,
we deduce the highest--order high frequency decay estimates. In this step,
we cannot work directly on the system of the variables $(\varrho_1, m_1, \varrho_2, m_2)$
as in \eqref{2eq-m}. Indeed, by noting the nonlinear terms in \eqref{N1N2-m}, we fail to
deal with the trouble terms involving $\nabla^{l+1}(\varrho_im_i)$ with $i=1,2$.
The main observation here is that instead of the variables $(\varrho_1, m_1, \varrho_2, m_2)$,
we study the system of the variables $(\varrho_1, u_1, \varrho_2, u_2)$ as in \eqref{2eq}. Then,
the corresponding trouble terms in \eqref{N1N2N3N4} become ones involving $\nabla^{l-1}(\varrho_i\Delta u_i)$ with $i=1,2$,
which however can be tackled with (see the proof of Lemma \ref{le-es-highorder1} for details).
\par
Forth, we prove the upper optimal decay rates of the solutions.
Combining the zero--order and the highest--order low frequency decay
estimates obtained in Step 2 with the highest--order high frequency
decay estimates obtained in Step 3, we can get the decay rates of
zero--order and the highest--order spatial derivatives of the
solutions by taking full advantage of the good properties of the
low--frequency and high--frequency decomposition. Then, by Sobolev
interpolation and the definition of $\mathcal{M}(t)$ in \eqref{M1},
we can get the key time--independent bound on $\mathcal{M}(t)$, and
this implies the upper optimal decay rates of the solutions in
Theorem \ref{1mainth} immediately.
\par
In the last step,
we show the lower optimal decay rates of the solutions.
To do this, we first employ Duhamel's principle, the lower decay rates of the
linear system in \eqref{low-li-es1} and \eqref{one-low-li-es1}, and Proposition \ref{li-low-de-es1}
to get the lower optimal decay rates of the solution as well as its one--order spatial derivative.
Then, for $1\leq k\leq l$, we can prove the lower optimal decay rates on the $k$--order spatial derivative
by an interpolation trick, and thus this completes the proof of the lower optimal decay rates in Theorem \ref{1mainth}.
\section{Reformulation of Original Problem}\label{Reformulation}
\noindent\textbf{2.1 Linearized System}

Let $\varrho_1 =\rho_1 -\frac1Z $ and $\varrho_2 = \rho_2 -1$. Then by using the fact that $\phi =\Delta^{-1}(Z\varrho_1 -\rho_2)$, the Cauchy problem \eqref{1eq}--\eqref{1id} can be rewritten as
\begin{equation}\label{2eq-m}
 \left\{\begin{array}{lll}
 \partial_t \varrho_1 + {\rm div}m_1 = 0,
 \\
 \partial_t m_1 + P'_1\left(\frac1Z \right)\nabla \varrho_1 - \nabla \Delta^{-1} (Z\varrho_1 - \varrho_2) -\mu_1Z \Delta m_1 -\nu_1Z \nabla {\rm div} m_1 = N^m_1,
 \\
 \partial_t \varrho_2 + {\rm div}m_2 =0,
 \\
 \partial_t m_2 + P'_2(1)\nabla \varrho_2 + \nabla \Delta^{-1} (Z\varrho_1 - \varrho_2) -\mu_2 \Delta m_2 -\nu_2 \nabla {\rm div} m_2 = N^m_2,
 \\
 (\varrho_1, m_1, \varrho_2, m_2)(x,0) =(\rho_1- \frac1Z, m_1, \rho_2-1, m_2)(x) := (\varrho_{10}, m_{10}, \varrho_{20}, m_{20})(x)
 \end{array}\right.
\end{equation}
with\small{
\begin{equation}\label{N1N2-m}
 \left\{\begin{array}{lll}
  \displaystyle N^m_1 = Z\varrho_1\nabla \phi -{\rm div}\mathbb{F}_1
  \\
  \displaystyle \qquad:= Z\varrho_1\nabla \phi -{\rm div}\left(\frac{m_1\otimes m_1}{\rho_1} +\left(P_1(\rho_1) -P_1\left(\frac1Z\right) -P_1'\left(\frac1Z\right)\rho_1\right)\mathbb{I}_3 +\mu_1 Z\nabla\left(\frac{\varrho_1 m_1}{\rho_1}\right)\right.
  \\
  \displaystyle \qquad\qquad\qquad\qquad\qquad\left.+\nu_1Z{\rm div}\left(\frac{\varrho_1 m_1}{\rho_1}\right) \mathbb{I}_3\right),
  \\
  \displaystyle N^m_2 = -\varrho_2\nabla \phi -{\rm div}\mathbb{F}_2
  \\
  \displaystyle \qquad:= -\varrho_2\nabla \phi -{\rm div}\left(\frac{m_2\otimes m_2}{\rho_2} +\left(P_2(\rho_2) -P_2(1) -P_2'(1)\rho_2\right)\mathbb{I}_3 +\mu_2 \nabla\left(\frac{\varrho_2 m_2}{\rho_2}\right)\right.
  \\
  \displaystyle \qquad\qquad\qquad\qquad\qquad\left.+\nu_2{\rm div}\left(\frac{\varrho_2 m_2}{\rho_2}\right) \mathbb{I}_3\right).
 \end{array}\right.
\end{equation}}
We will do the decay rate on the lower--frequent part of the solution to the system \eqref{2eq-m}. Exactly speaking, thanks to the divergent form of the nonlinear terms in \eqref{N1N2-m}, we can estimate the convergence rate of the solution by shifting differential operator. On the contrary, we would fail to estimate the decay rate on the higher--frequent part of the highest--order derivatives of the solution due to the divergent form. To this end, we should introduce the following system of the densities $\varrho_i$ and velocities $u_i= \frac{m_i}{\rho_i}$ with $i = 1,2$
\begin{equation}\label{2eq}
 \left\{\begin{array}{lll}
  \partial_t \varrho_1 + \frac1Z{\rm div}u_1 = N^\varrho_1,
  \\
  \partial_t u_1 + P'_1(1)\nabla \varrho_1 - Z\nabla \Delta^{-1} (Z\varrho_1 - \varrho_2) -\mu_1Z \Delta u_1 -\nu_1Z \nabla {\rm div} u_1 = N^u_1,
  \\
  \partial_t \varrho_2 + {\rm div}u_2 =N^\varrho_2,
  \\
  \partial_t u_2 + P'_2(1)\nabla \varrho_2 + \nabla \Delta^{-1} (Z\varrho_1 - \varrho_2) -\mu_2 \Delta u_2 -\nu_2 \nabla {\rm div} u_2 = N^u_2,
  \\
  (\varrho_1, u_1, \varrho_2, u_2)(x,0) =(\varrho_{10}, \frac{m_{10}}{\rho_{10}}, \varrho_{20}, \frac{m_{20}}{\rho_{20}})(x) := (\varrho_{10}, u_{10}, \varrho_{20}, u_{20})(x)
 \end{array}\right.
\end{equation}
with
\begin{equation}\label{N1N2N3N4}
 \left\{\begin{array}{lll}
  \displaystyle N^\varrho_1 = -{\rm div}(\varrho_1u_1),
  \\
  \displaystyle N^u_1 = -u_1\cdot \nabla u_1 - \left( \frac {P'_1(\rho_1)} {\rho_1}-ZP'_1\left(\frac1Z\right) \right)\nabla \varrho_1  -\frac {\mu_1 Z\varrho_1}{\rho_1}\Delta u_1 - \frac {\nu_1 Z\varrho_1}{\rho_1} \nabla {\rm div} u_1,
  \\
  N^\varrho_2 = -{\rm div}(\varrho_2u_2),
  \\
  \displaystyle N^u_2 = -u_2\cdot \nabla u_2 - \left( \frac {P'_2(\rho_2)} {\rho_2}-P'_2(1) \right)\nabla \varrho_2 -\frac {\mu_2 \varrho_2}{\rho_2}\Delta u_2 - \frac {\nu_2 \varrho_2}{\rho_2} \nabla {\rm div} u_2.
 \end{array}\right.
\end{equation}

\noindent\textbf{2.2 ``div--curl" Decomposition}

For $i = 1, 2$, let $n_i = \Lambda^{-1} {\rm div} m_i$ be the ``compressible part" of $m_i$ and $M_i = \Lambda^{-1} {\rm curl} m_i$ (with ${\rm curl}z = (\partial_{x_2}z^3 -\partial_{x_3}z^2, \partial_{x_3}z^1 -\partial_{x_1}z^3, \partial_{x_1}z^2 -\partial_{x_2}z^1)^T$) be the ``incompressible part" of $m_i$ respectively. Then the system \eqref{2eq-m} can be rewritten as two parts in the following
\begin{equation}\label{3eq}
 \left\{\begin{array}{lll}
  \partial_t \varrho_1 + \Lambda n_1 = 0,
  \\
  \partial_t n_1 - P'_1\left(\frac1Z\right)\Lambda \varrho_1 - \Lambda^{-1} (Z\varrho_1 - \varrho_2) -(\mu_1 + \nu_2)Z \Delta n_1 = \Lambda^{-1} {\rm div} N^m_1,
  \\
  \partial_t \varrho_2 + \Lambda n_2 = 0,
  \\
  \partial_t n_2 - P'_2(1)\Lambda \varrho_2 + \Lambda^{-1} (Z\varrho_1 - \varrho_2) - (\mu_2 + \nu_2) \Delta n_2 = \Lambda^{-1} {\rm div} N^m_2,
  \\
  (\varrho_1, n_1, \varrho_2, n_2)(x,0) = (\varrho_{10}, \Lambda^{-1}{\rm div}m_{10}, \varrho_{20}, \Lambda^{-1}{\rm div}m_{20})(x) := (\varrho_{10}, n_{10}, \varrho_{20}, n_{20})(x)
 \end{array}\right.
\end{equation}
and
\begin{equation}\label{4eq}
 \left\{\begin{array}{lll}
  \partial_t M_1 - \mu_1Z \Delta M_1 = \Lambda^{-1} {\rm curl} N^m_1,
  \\
  \partial_t  M_2 - \mu_2 \Delta  M_2 = \Lambda^{-1} {\rm curl} N^m_2,
  \\
  (M_1,  M_2)(x,0) = (\Lambda^{-1}{\rm curl} m_{10}, \Lambda^{-1}{\rm curl}m_{20})(x) :=(M_{10}, M_{20})(x).
 \end{array}\right.
\end{equation}

Note that \eqref{3eq} are hyperbolic--parabolic system that the structure of the solution semigroup is simpler than one of \eqref{2eq-m}, and \eqref{4eq} are mere heat equations on the $M_i$. Moreover, by the relationship
\begin{equation}\nonumber
 m_i = -\Lambda^{-1}\nabla n_i - \Lambda^{-1}{\rm div}M_i, \quad i=1,2
\end{equation}
involving pseudo--differential operators of degree zero, the estimates in the space $H^l$ for the original function $m_i$ can be derived from $n_i$ and $M_i$. Hence we will focus on the spectral analysis on the solution semigroups of \eqref{3eq}--\eqref{4eq}.

\noindent\textbf{2.3 Spectral Analysis}

Let $U = (\varrho_1, n_1, \varrho_2, n_2)^T$. Due to the semigroup theory for evolutionary equation, we will study the following initial value problem for the linear system
\begin{equation}\label{5eq}
 \left\{\begin{array}{lll}
  U_t = \mathbb{B}U,
  \\
  U|_{t=0} = U_0 = (\varrho_{10}, n_{10}, \varrho_{20}, n_{20})^T,
 \end{array}\right.
\end{equation}
where the operator $\mathbb{B}$ is given by
\begin{equation}\nonumber
\mathbb{B} = \left(\begin{array}{cccc}
 0 &   -\Lambda & 0 & 0
 \\
 Z\Lambda^{-1} + P_1'\left(\frac1Z\right)\Lambda & (\mu_1 + \nu_1)Z\Delta & -\Lambda^{-1} & 0
 \\
 0 & 0 & 0 & -\Lambda
 \\
 -Z\Lambda^{-1} & 0 & \Lambda^{-1} + P_2'(1)\Lambda & (\mu_2 +\nu_2)\Delta
 \end{array}\right).
\end{equation}

Taking the Fourier transform to  the system, we have
\begin{equation}\nonumber
 \left\{\begin{array}{lll}
  \widehat{U}_t = \mathbb{A}(\xi)\widehat{U},
  \\
  \widehat{U}|_{t=0} = \widehat{U}_0,
 \end{array}\right.
\end{equation}
where $\widehat{U}(\xi, t) = \mathfrak{F}(U(x, t))$ and $\mathbb{A}(\xi)$ is given by
\begin{equation}\nonumber
 \mathbb{A}(\xi) = \left(\begin{array}{cccc}
 0 &   -|\xi| & 0 & 0
 \\
 Z|\xi|^{-1} + P_1'\left(\frac1Z\right)|\xi| & -(\mu_1 + \nu_1)Z|\xi|^2 & -|\xi|^{-1} & 0
 \\
 0 & 0 & 0 & -|\xi|
 \\
 -Z|\xi|^{-1} & 0 & |\xi|^{-1} + P_2'(1)|\xi| & -(\mu_2 +\nu_2)|\xi|^2
 \end{array}\right).
\end{equation}

The eigenvalues of the matrix $\mathbb{A}(\xi)$ can be solved from the determinant\small{
\begin{align*}
  \det\{\mathbb{A}(\xi) - \lambda \mathbb{I}\}
  &= \lambda ^4 +(Z(\mu_1 +\nu_1) +\mu_2 +\nu_2)|\xi|^2 \lambda^3
  \\
  &\qquad + \left(Z(\mu_1 +\nu_1)(\mu_2 +\nu_2)|\xi|^4 + \left( P_1'\left(\frac1Z\right) + P_2'(1)\right) |\xi|^2 +1+Z\right)\lambda^2 \\
  &\qquad +\left(\left(Z(\mu_1 +\nu_1)P_2'(1) +(\mu_2 +\nu_2)P_1'\left(\frac1Z\right)\right)|\xi|^4 +Z(\mu_1 +\nu_1 +\mu_2 +\nu_2)|\xi|^2\right)\lambda
  \\
  & \qquad + P_1'\left(\frac1Z\right) P_2'(1) |\xi|^4 + \left( P_1'\left(\frac1Z\right) + ZP_2'(1)\right) |\xi|^2
  \\
  &=0.
\end{align*}}
\!\!By direct calculation and delicate analysis on the roots of the above quartic equation, we can deduce that the eigenvalues of the matrix $\mathbb{A}(\xi)$ has four different eigenvalues $\lambda_i = \lambda_i (\xi)$ with $i= 1, 2, 3, 4$ while $|\xi|\ll1$; the detail also can be seen in \cite{CY}. Hence we can decompose the semigroup $e^{t\mathbb{A}(\xi)}$ in the following:
\begin{equation}\nonumber
 e^{t\mathbb{A}(\xi)}
 = \sum_{i = 1}^4e^{\lambda_i t} \mathbb{P}_i(\xi)
\end{equation}
with the projector $\mathbb{P}_i(\xi)$ given by
\begin{equation}\nonumber
 \mathbb{P}_i(\xi)
 = \prod_{j\neq i}\frac{ \mathbb{A}(\xi) - \lambda_j I}{\lambda_i - \lambda_j},\quad i,j=1,2,3,4
\end{equation}

Then we can represent the solution of the problem as
\begin{equation}\label{so-expr1}
 \widehat{U}(\xi, t)
 = e^{t\mathbb{A}(\xi)} \widehat{U}_0(\xi)
 = \left(\sum_{i=1}^4 e^{\lambda_i t}\mathbb{P}_i(\xi)\right) \widehat{U}_0(\xi).
\end{equation}

\begin{lemma}\label{eigenvalues}
 There exists a positive constant $\eta\ll1$ such that, for $|\xi|\le \eta$, the spectral has the following Taylor series expansion
 \begin{equation}\nonumber
  \left\{\begin{array}{lll}\displaystyle
   \lambda_{1,2}
   = -\kappa_1|\xi|^2 +O(|\xi|^4) \pm i\left(\sqrt{1+Z} +\frac{\sigma_1^2}{2\sqrt{1+Z}}|\xi|^2 +O(|\xi|^4)\right)
   \\
   \displaystyle\lambda_{3,4}
   = -\kappa_2|\xi|^2  +O(|\xi|^4) \pm i\left(\sigma_2|\xi| +O(|\xi|^3)\right)
  \end{array}\right.
 \end{equation}
 with $\kappa_1= \frac{Z^2(\mu_1 +\nu_2) +\mu_2 +\nu_2}{2(1+Z)}$, $\kappa_2= \frac{Z(\mu_1 +\nu_2 +\mu_2 +\nu_2)}{2(1+Z)}$, $\sigma_1 = \sqrt{\frac{ZP_1'\left(\frac1Z\right) +P_2'(1)}{1+Z}}$, and $\sigma_2 = \sqrt{\frac{P_1'\left(\frac1Z\right) +ZP_2'(1)}{1+Z}}$.
\end{lemma}

We establish the following estimates for the low--frequent part of the solutions $\widehat{U}(\xi, t)$ to the problem \eqref{3eq} and \eqref{4eq} while $N_i = 0$ with $i = 1,2,3,4$:
\begin{lemma}\label{eigenvalues}
 (i) For $|\xi|\le \eta$, we have
 \begin{equation}\nonumber
  \begin{split}
   &|\widehat\varrho_1(\xi, t)|, \quad |\widehat\varrho_2(\xi, t)|
   \lesssim \left(e^{- \frac{\kappa_1}2 |\xi|^2t} +e^{- \frac{\kappa_2}2 |\xi|^2t}\right) |\widehat{U}_0(\xi)|,
   \\
   &|\widehat n_1(\xi, t)|
   \lesssim e^{- \frac{\kappa_1}2 |\xi|^2t} \frac {|\sin (t{\rm Im} \lambda_1)| }{\sqrt{1+Z}} |\xi|^{-1}|Z\widehat{\varrho_{10}}(\xi) - \widehat{\varrho_{20}}(\xi)| +\left(e^{- \frac{\kappa_1}2 |\xi|^2t} +e^{- \frac{\kappa_2}2 |\xi|^2t} \right)|\widehat{U}_0(\xi)|,
  \end{split}
 \end{equation}
 \begin{equation}\nonumber
  |\widehat n_2(\xi, t)|
  \lesssim e^{- \frac{\kappa_1}2 |\xi|^2t} \frac {|\sin (t{\rm Im} \lambda_1)|}{\sqrt{1+Z}} |\xi|^{-1}||Z\widehat{\varrho_{10}}(\xi) - \widehat{\varrho_{20}}(\xi)| +\left(e^{- \frac{\kappa_1}2 |\xi|^2t} +e^{- \frac{\kappa_2}2 |\xi|^2t} \right)|\widehat{U}_0(\xi)|.
 \end{equation}
 and
 \begin{equation}\nonumber
  |\widehat{\nabla \phi} (\xi, t)|
  \lesssim e^{-\frac{\kappa_1}2 |\xi|^2t}|\xi|^{-1}|Z\widehat{\varrho_{10}}(\xi) - \widehat{\varrho_{20}}(\xi)| + \left(e^{- \frac{\kappa_1}2 |\xi|^2t} +e^{- \frac{\kappa_2}2 |\xi|^2t} \right)|\widehat{U}_0(\xi)|.
 \end{equation}

 (ii) For any $\xi$, we have
 \begin{equation}\nonumber
  |\widehat M_1(\xi, t)|
  \sim e^{- \mu_1|\xi|^2 t}|\widehat{ M_{10}}(\xi)|
 \end{equation}
 and
 \begin{equation}\nonumber
   |\widehat{ M}_2(\xi, t)|
   \sim e^{- \mu_2|\xi|^2 t}|\widehat{ M_{20}}(\xi)|.
 \end{equation}
\end{lemma}
\begin{proof}
 Part $(ii)$ can be easily derived from the standard heat equation for $ M_1$ and $  M_2$.
 In order to prove part $(i)$, we express $\mathbb{P}_i(i=1,2,3,4)$ as
 \begin{equation}\nonumber
  \mathbb{P}_1(\xi)
  =\begingroup
   \renewcommand*{\arraystretch}{1.5}
   \left(\begin{array}{cccc}
   \frac{Z}{2(1+Z)} & 0 & -\frac{1}{2(1+Z)} & 0
   \\
   -\frac{iZ|\xi|^{-1}}{2\sqrt{1+Z}} & \frac{Z}{2(1+Z)} & \frac{i|\xi|^{-1}}{2\sqrt{1+Z}} & -\frac{1}{2(1+Z)}
   \\
   -\frac{Z}{2(1+Z)} & 0 & \frac{1}{2(1+Z)} & 0
   \\
   \frac{iZ|\xi|^{-1}}{2\sqrt{1+Z}} & -\frac{Z}{2(1+Z)} & -\frac{i|\xi|^{-1}}{2\sqrt{1+Z}} & \frac{1}{2(1+Z)}
   \end{array}\right)
  \endgroup + \left(O(|\xi|) +iO(|\xi|)\right)\mathbb{J},
 \end{equation}

 \begin{equation}\nonumber
  \mathbb{P}_2(\xi)
  =\begingroup
   \renewcommand*{\arraystretch}{1.5}
   \left(\begin{array}{cccc}
   \frac{Z}{2(1+Z)} & 0 & -\frac{1}{2(1+Z)} & 0
   \\
   \frac{iZ|\xi|^{-1}}{2\sqrt{1+Z}} & \frac{Z}{2(1+Z)} & -\frac{i|\xi|^{-1}}{2\sqrt{1+Z}} & -\frac{1}{2(1+Z)}
   \\
   -\frac{Z}{2(1+Z)} & 0 & \frac{1}{2(1+Z)} & 0
   \\
   -\frac{iZ|\xi|^{-1}}{2\sqrt{1+Z}} & -\frac{Z}{2(1+Z)} & \frac{i|\xi|^{-1}}{2\sqrt{1+Z}} & \frac{1}{2(1+Z)}
   \end{array}\right)
  \endgroup + \left(O(|\xi|) +iO(|\xi|)\right)\mathbb{J},
 \end{equation}

 \begin{equation}\nonumber
  \mathbb{P}_3(\xi)
  =\begingroup
   \renewcommand*{\arraystretch}{1.5}
   \left(\begin{array}{cccc}
   \frac1{2(1+Z)} & \frac{i}{2(1+Z)\sigma_2} & \frac 1{2(1+Z)} & \frac{i}{2(1+Z)\sigma_2}
   \\
   -\frac{i\sigma_2}{2(1+Z)} & \frac1{2(1+Z)} & -\frac{i\sigma_2}{2(1+Z)} & \frac 1{2(1+Z)}
   \\
   \frac Z{2(1+Z)} & \frac{iZ}{2(1+Z)\sigma_2} & \frac Z{2(1+Z)} & \frac{iZ}{2(1+Z)\sigma_2}
   \\
   -\frac{iZ\sigma_2}{2(1+Z)} & \frac Z{2(1+Z)} & -\frac{iZ\sigma_2}{2(1+Z)} & \frac Z{2(1+Z)}
   \end{array}\right)
  \endgroup + \left(O(|\xi|) +iO(|\xi|)\right)\mathbb{J},
 \end{equation}
 and
 \begin{equation}\nonumber
  \mathbb{P}_4(\xi)
  =\begingroup
   \renewcommand*{\arraystretch}{1.5}
   \left(\begin{array}{cccc}
   \frac1{2(1+Z)} & -\frac{i}{2(1+Z)\sigma_2} & \frac 1{2(1+Z)} & -\frac{i}{2(1+Z)\sigma_2}
   \\
   \frac{i\sigma_2}{2(1+Z)} & \frac1{2(1+Z)} & \frac{i\sigma_2}{2(1+Z)} & \frac 1{2(1+Z)}
   \\
   \frac Z{2(1+Z)} & -\frac{iZ}{2(1+Z)\sigma_2} & \frac Z{2(1+Z)} & -\frac{iZ}{2(1+Z)\sigma_2}
   \\
   \frac{iZ\sigma_2}{2(1+Z)} & \frac Z{2(1+Z)} & \frac{iZ\sigma_2}{2(1+Z)} & \frac Z{2(1+Z)}
   \end{array}\right)
  \endgroup + \left(O(|\xi|) +iO(|\xi|)\right)\mathbb{J},
 \end{equation}
 where $\mathbb{J}$ is a 4--order matrix with all elements equal to 1.

 Then we can conclude that
 \begin{equation}\label{so-max1}
  \begin{split}
  \sum_{i = 1}^4e^{\lambda_i t} \mathbb{P}_i(\xi)
  =&\begingroup
   \renewcommand*{\arraystretch}{1.8}
   \left(\begin{array}{cccc}
   \frac {Zg^{1,2}_+ + g^{3,4}_+}{2(1+Z)}  & \frac{ig^{3,4}_-}{2(1+Z)\sigma_2} & \frac {g^{3,4}_+ - g^{1,2}_+}{2(1+Z)} & \frac{ig^{3,4}_-}{2(1+Z)\sigma_2}
   \\
   -\frac{iZ|\xi|^{-1}g^{1,2}_-}{2\sqrt{1+Z}} - \frac{i\sigma_2g^{3,4}_-}{2(1+Z)} & \frac {Zg^{1,2}_+ + g^{3,4}_+}{2(1+Z)} & \frac{i|\xi|^{-1}g^{1,2}_-}{2\sqrt{1+Z}} - \frac{i\sigma_2g^{3,4}_-}{2(1+Z)} & \frac {g^{3,4}_+ - g^{1,2}_+}{2(1+Z)}
   \\
   \frac {Z(g^{3,4}_+ - g^{1,2}_+)}{2(1+Z)} & \frac{iZg^{3,4}_-}{2(1+Z)\sigma_2} & \frac {g^{1,2}_+ + Zg^{3,4}_+}{2(1+Z)} & \frac{iZg^{3,4}_-}{2(1+Z)\sigma_2}
   \\
   \frac{iZ|\xi|^{-1}g^{1,2}_-}{2\sqrt{1+Z}} - \frac{iZ\sigma_2g^{3,4}_-}{2(1+Z)} & \frac {Z(g^{3,4}_+ - g^{1,2}_+)}{2(1+Z)} & -\frac{i|\xi|^{-1}g^{1,2}_-}{2\sqrt{1+Z}} - \frac{iZ\sigma_2g^{3,4}_-}{2(1+Z)} & \frac {g^{1,2}_+ + Zg^{3,4}_+}{2(1+Z)}
   \end{array}\right)
   \endgroup
   \\
   &\qquad+\left(O(|\xi|) +iO(|\xi|)\right)\left(e^{- \kappa_1 |\xi|^2t +O(|\xi|^4)t} +e^{- \kappa_2 |\xi|^2t +O(|\xi|^4)t}\right) \mathbb{J}
  \end{split}
 \end{equation}
 with
 \begin{equation}\nonumber
  g_\pm^{i,j}
  = e^{t\lambda_i} \pm e^{t\lambda_j}, \quad for \quad i,j =1,2,3,4.
 \end{equation}
 By plugging \eqref{so-max1} to \eqref{so-expr1}, we can complete the proof of part $(i)$.
\end{proof}

\noindent\textbf{2.4 Upper Decay Rate for the Linear System}

Thanks to Lemma \ref{eigenvalues}, we can estimate the decay rates on the lower--frequency of the solutions to the linear systems \eqref{3eq} and \eqref{4eq} while $N^m_1= N^m_2 = 0$ as follows:
\begin{Proposition}\label{li-de-es1}
 For any $1\le p\le 2$ and $k=0,1, \cdots, l$, there exists a positive constant $C$ which is independent of $t$ such that
 \begin{equation}\label{li-varrho-1}
  \|\nabla^k( \varrho_1 ^L, \varrho_2^L )\|_{L^2}
  \le C (1+t)^{- \frac34 -\frac k2}\|U_0^L\|_{L^1} ,
 \end{equation}
 \begin{equation}\label{li-v-1}
  \|\nabla^k( n_1 ^L, n_2^L, (\nabla\phi)^L)\|_{L^2}
  \le C (1+t)^{- \frac34 -\frac k2}\|U_0^L\|_{L^1} +C (1+t)^{- \frac32(\frac1p-\frac12) -\frac k2} \|(\nabla \phi_0)^L\|_{L^p},
 \end{equation}
 and
 \begin{equation}\label{li-omega-1}
  \|\nabla^k(  M_1,  M_2)\|_{L^2}
  \le C (1+t)^{- \frac34 -\frac k2} \|U_0\|_{L^1}.
 \end{equation}
 \begin{proof}
  Here we only prove the decay on $\nabla^k(\nabla\phi)^L$ for $1<p<2$. The other cases can be proved in the similar way or the proof can be seen in our previous work \cite{C1, WZZ2}. Indeed, by Lemma \ref{eigenvalues}, Plancherel theorem and Hausdorff--Young's inequality, we have that for each $0\le k\le l$ and $|\xi|\le \eta$,
  \begin{equation}\nonumber
   \begin{split}
   \|\nabla^k(\nabla\phi)^L\|_{L^2}^2
   &=\||\xi|^k\widehat{(\nabla\phi)^L}(\xi)\|_{L^2}^2
   \\
   &\lesssim \int_{|\xi|\le \eta}|\xi|^{2k}\left(e^{-\frac{\kappa_1}2 |\xi|^2t}|\xi|^{-1}|Z\widehat{\varrho_{10}}(\xi) - \widehat{\varrho_{20}}(\xi)| + \left(e^{- \frac{\kappa_1}2 |\xi|^2t} +e^{- \frac{\kappa_2}2 |\xi|^2t} \right)|\widehat{U}_0(\xi)|\right)^2d\xi
   \\
   &\lesssim \left\||\xi|^{-1}(Z\widehat{\varrho_{10}}(\xi) -\widehat{\varrho_{20}}(\xi))\right\|_{L^q}^2 \left(\int_{|\xi|\le \eta}|\xi|^{\frac{2qk}{q-2}}e^{-\frac{q\kappa_1}{q-2}|\xi|^2t}d\xi \right)^{\frac{q-2}q}
   \\
   &\qquad + \|\widehat{U_0}\|_{L^\infty(|\xi|\le\eta)}^2 \int_{|\xi|\le\eta}|\xi|^{2k}
   \left(e^{-\kappa_1|\xi|^2t} +e^{-\kappa_2|\xi|^2t}\right)d\xi
   \\
   &\lesssim (1+t)^{- 3\left(\frac12 - \frac1q\right) - k}\|\widehat{\nabla\phi_0}\|_{L^q(|\xi|\le\eta)}^2 +(1+t)^{-\frac 32 -k}\|\widehat{U_0}\|_{L^\infty(|\xi|\le\eta)}^2
   \\
   &\lesssim (1+t)^{- 3\left(\frac1p - \frac12\right) - k}\|(\nabla\phi_0)^L\|_{L^p}^2 +(1+t)^{-\frac 32 -k}\|U_0^L\|_{L^1}^2,
   \end{split}
  \end{equation}
  where $\frac1p +\frac1q = 1$.
 \end{proof}
\end{Proposition}

In order to estimate the convergence rate for the system \eqref{2eq-m}, we also need to analyze the decay of the nonlinear terms. To this end, we rewrite the system \eqref{3eq} as
\begin{equation}\label{5eq-noli}
 \left\{\begin{array}{lll}
  U_t = BU + \mathcal{N},
  \\
  U|_{t=0} = U_0,
 \end{array}\right.
\end{equation}
with
\begin{equation}\nonumber
 \mathcal{N}
 = (0, \Lambda^{-1}{\rm div}N^m_1, 0, \Lambda^{-1}{\rm div}N^m_2)^T.
\end{equation}
Then the solution of \eqref{5eq-noli} can be expressed as
\begin{equation}\label{so-ex}
 U =e^{tB}\ast U_0 + \int_0^t e^{(t-\tau)B}\ast \mathcal{N}(\tau)d\tau .
\end{equation}

Define
\begin{equation}\nonumber
 \mathcal{S}(x, \tau)
 := (S^\varrho_1(x, \tau), S^n_1(x, \tau), S^\varrho_2(x, \tau), S^n_2(x, \tau))^T
 = e^{(t-\tau)B}\ast \mathcal{N}(\tau)
\end{equation}
and its Fourier transform
\begin{equation}\nonumber
 \mathfrak{F}[\mathcal{S}(\tau)]
 := (\widehat{S^\varrho_1}(\xi, \tau), \widehat{S^n_1}(\xi, \tau), \widehat{S^\varrho_2}(\xi, \tau), \widehat{S^n_2}(\xi, \tau))^T.
\end{equation}

Now we complement the decay estimates on the nonlinear term of the expression \eqref{so-ex} of the solution $U(x, t)$ in the following:
\begin{Proposition}\label{prop-decay4}
 It holds for $k=0,1, \cdots, l$ that
 \begin{equation}\label{decay-S-varrho}
  \|\nabla ^k\left((S^\varrho_1)^L, (S^\varrho_2)^L\right)(\tau)\|_{L^2}
  \lesssim (1+t-\tau)^{- \frac54 -\frac k2}(\|(\nabla\phi)^L(\tau)\|_{L^2}^2 +\|(\mathcal{N}^L, \mathbb{F}_1^L, \mathbb{F}_2^L)(\tau)\|_{L^1}),
 \end{equation}
 \begin{equation}\label{decay-S-varrho2}
  \|\nabla ^k\left((S^\varrho_1)^L, (S^\varrho_2)^L\right)(\tau)\|_{L^2}
  \lesssim (1+t-\tau)^{- \frac12}\|\nabla^k\mathcal{N}^L(\tau)\|_{L^2},
 \end{equation}
 \begin{equation}\label{decay-S-n}
  \|\nabla ^k\left((S^\varrho_1)^L, (S^\varrho_2)^L, (S^n_1)^L, (S^n_2)^L, \nabla\Delta^{-1}\left(Z(S^\varrho_1)^L- (S^\varrho_2)^L\right)\right)(\tau)\|_{L^2}
  \lesssim (1+t-\tau)^{- \frac34 -\frac k2}\|\mathcal{N}^L(\tau)\|_{L^1},
 \end{equation}
 and
 \begin{equation}\label{decay-S-n2}
  \|\nabla^k\left((S^\varrho_1)^L, (S^\varrho_2)^L, (S^n_1)^L, (S^n_2)^L, \nabla\Delta^{-1}\left(Z(S^\varrho_1)^L- (S^\varrho_2)^L\right)\right)(\tau)\|_{L^2}
  \lesssim (1+t-\tau)^{- \frac34}\|\nabla^k\mathcal{N}^L(\tau)\|_{L^1}.
 \end{equation}

 \begin{proof}
  We can derive from the expression \eqref{so-max1} of $e^{t\mathbb{A}(\xi)}$ that for $|\xi|\le \eta$,
  \begin{equation}\label{S3}
   \begin{split}
   \mathfrak{F}[\mathcal{S}(\tau)]
   &=\begingroup
     \renewcommand*{\arraystretch}{2.2}\left(\begin{array}{c}
     \displaystyle \frac{ig^{3,4}_-}{2(1+Z)\sigma_2}\mathfrak{F}\left[\Lambda^{-1}{\rm div}(Z\varrho_1\nabla\phi - {\rm div}\mathbb{F}_1 -\varrho_2\nabla\phi - {\rm div}\mathbb{F}_2)\right]
     \\
     \displaystyle \frac{Zg^{1,2}_+ +g^{3,4}_+}{2(1+Z)}\mathfrak{F}[\Lambda^{-1}{\rm div}N^m_1] +\frac{g^{3,4}_+ -g^{1,2}_+}{2(1+Z)}\mathfrak{F}[\Lambda^{-1}{\rm div}N^m_2]
     \\
     \displaystyle \frac{iZg^{3,4}_-}{2(1+Z)\sigma_2}\mathfrak{F}\left[\Lambda^{-1}{\rm div}(Z\varrho_1\nabla\phi - {\rm div}\mathbb{F}_1 -\varrho_2\nabla\phi - {\rm div}\mathbb{F}_2)\right]
     \\
     \displaystyle \frac{Z(g^{3,4}_+ -g^{1,2}_+)}{2(1+Z)}\mathfrak{F}[\Lambda^{-1}{\rm div}N^m_1] +\frac{ g^{1,2}_+ +Zg^{3,4}_+}{2(1+Z)}\mathfrak{F}[\Lambda^{-1}{\rm div}N^m_2]
     \end{array}\right)
     \endgroup
     \\
     &\displaystyle \qquad +\left(O(|\xi|) +iO(|\xi|)\right)\left(e^{- \kappa_1 |\xi|^2(t-\tau) +O(|\xi|^4)(t-\tau)} +e^{- \kappa_2 |\xi|^2(t-\tau) +O(|\xi|^4)(t-\tau)}\right)\mathbb{J} \widehat{\mathcal{N}}
     \\
     &=\begingroup
     \renewcommand*{\arraystretch}{2.2}\left(\begin{array}{c}
     \displaystyle \frac{ig^{3,4}_-}{2(1+Z)\sigma_2}\mathfrak{F}\left[\Lambda^{-1}{\rm div}{\rm div}\left( \nabla\phi\otimes\nabla\phi - \frac{|\nabla\phi|^2}2\mathbb{I} - \mathbb{F}_1- \mathbb{F}_2\right)\right]
     \\
     \displaystyle \frac{Zg^{1,2}_+ +g^{3,4}_+}{2(1+Z)}\mathfrak{F}[\Lambda^{-1}{\rm div}N^m_1] +\frac{g^{3,4}_+ -g^{1,2}_+}{2(1+Z)}\mathfrak{F}[\Lambda^{-1}{\rm div}N^m_2]
     \\
     \displaystyle \frac{iZg^{3,4}_-}{2(1+Z)\sigma_2}\mathfrak{F} \left[\Lambda^{-1}{\rm div}{\rm div}\left( \nabla\phi\otimes\nabla\phi - \frac{|\nabla\phi|^2}2\mathbb{I} - \mathbb{F}_1- \mathbb{F}_2\right)\right]
     \\
     \displaystyle \frac{Z(g^{3,4}_+ -g^{1,2}_+)}{2(1+Z)}\mathfrak{F}[\Lambda^{-1}{\rm div}N^m_1] +\frac{ g^{1,2}_+ +Zg^{3,4}_+}{2(1+Z)}\mathfrak{F}[\Lambda^{-1}{\rm div}N^m_2]
     \end{array}\right),
     \endgroup
     \\
     &\displaystyle \qquad+\left(O(|\xi|) +iO(|\xi|)\right)\left(e^{- \kappa_1|\xi|^2(t-\tau) +O(|\xi|^4)(t-\tau)} +e^{- \kappa_2 |\xi|^2(t-\tau) +O(|\xi|^4)(t-\tau)}\right) \mathbb{J}\widehat{\mathcal{N}}
     \\
   \end{split}
  \end{equation}
  where $\mathbb{F}_1$ and $\mathbb{F}_2$ are defined in \eqref{N1N2-m}. Hence by \eqref{S3}, Plancherel theorem and Hausdorff--Young's inequality, we have that for each $0\le k\le l$ and $|\xi|\le \eta$,
  \begin{equation}\nonumber
   \begin{split}
   &\|\nabla^k\left((S^\varrho_1)^L, (S^\varrho_2)^L\right)\|_{L^2}^2
   \\
   & =\||\xi|^k\left(\mathfrak{F}[(S^\varrho_1)^L], \mathfrak{F}[(S^\varrho_2)^L]\right)\|_{L^2}^2
   \\
   &\lesssim \int_{|\xi|\le\eta}|\xi|^{2k}e^{- 2\kappa_2 |\xi|^2(t-\tau) +O(|\xi|^4)(t-\tau)} |\xi|^2\mathfrak{F}\left[\left( \nabla\phi\otimes\nabla\phi - \frac{|\nabla\phi|^2}2\mathbb{I} - \mathbb{F}_1- \mathbb{F}_2\right)\right]^2(\tau) d\xi
   \\
   & \qquad +\int_{|\xi|\le\eta}|\xi|^{2k}\left(O(|\xi|) +iO(|\xi|)\right)^2\left(e^{- \kappa_1 |\xi|^2(t-\tau) +O(|\xi|^4)(t-\tau)} +e^{- \kappa_2 |\xi|^2(t-\tau) +O(|\xi|^4)(t-\tau)}\right)^2 |\widehat{\mathcal{N}}(\tau)|^2 d\xi
   \\
   &\lesssim (1+ t- \tau)^{-\frac52 -k}\left(\Bigg\|\mathfrak{F} \left[\left( \nabla\phi\otimes\nabla\phi - \frac{|\nabla\phi|^2}2\mathbb{I} - \mathbb{F}_1- \mathbb{F}_2\right)\right](\tau)\Bigg\|_{L^\infty(|\xi|\le \eta)}^2 +\big\|\widehat{\mathcal{N}}(\tau)\big\|_{L^\infty(|\xi|\le \eta)}^2\right)
   \\
   &\lesssim (1+ t- \tau)^{-\frac52 -k}\left(\Bigg\|\left(\nabla\phi\otimes\nabla\phi - \frac{|\nabla\phi|^2}2\mathbb{I} - \mathbb{F}_1- \mathbb{F}_2\right)^L(\tau)\Bigg\|_{L^1}^2 +\|\mathcal{N}^L(\tau)\|_{L^1}^2\right),
   \end{split}
  \end{equation}
  which gets rise to \eqref{decay-S-varrho}. Obviously, \eqref{decay-S-varrho2}, \eqref{decay-S-n} and \eqref{decay-S-n2} can be obtained in the similar way.
 \end{proof}
\end{Proposition}

\noindent\textbf{2.5 Lower Decay Rate for the Linear System}

The lower--bounds on the decay rates for the above linear system are given in the following proposition:
\begin{Proposition}\label{li-low-de-es1}
 Assume that $(\varrho_{10}, n_{10}, \varrho_{20}, n_{20}, \nabla\phi_0) \in L^1$ satisfies
 \begin{equation}\nonumber
  \widehat{\varrho_{10}}(\xi) = \widehat{n_{10}}(\xi) = \widehat{ M_{10}}(\xi) = \widehat{\varrho_{20}}(\xi) =\widehat{ M_{10}}(\xi) = 0, \quad and \quad |\widehat{n_{20}}| \geq C\delta_0^\frac32
 \end{equation}
 for any $|\xi|\le \eta$. Then there exists a positive constant $C_2$, which is independent of $t$, such that the global solution $(\varrho_{1}, n_{1}, \varrho_{2}, n_{2}) $ of the IVP \eqref{5eq} satisfies
 \begin{equation}\label{low-li-es1}
  \min\left\{\|\varrho_{1}^L\|_{L^2}, \|n_{1}^L\|_{L^2}, \|\varrho_{2}^L\|_{L^2}, \|n_{2}^L\|_{L^2}\right\}\geq C_2\delta_0^\frac32(1+t)^{-\frac34},
 \end{equation}
 and
 \begin{equation}\label{one-low-li-es1}
  \min\left\{\|\nabla\varrho_{1}^L\|_{L^2}, \|\nabla n_{1}^L\|_{L^2}, \|\nabla\varrho_{2}^L\|_{L^2}, \|\nabla n_{2}^L\|_{L^2}\right\}
  \geq C_2\delta_0^\frac32(1+t)^{-\frac54}
 \end{equation}
 for any large enough $t$.
\end{Proposition}
\begin{proof}
 Here we only prove the lower decay rate on $n_2$. From \eqref{so-expr1} and \eqref{so-max1}, we have that for $|\xi|\le \eta$,
 \begin{equation}\nonumber
  \begin{split}
  &\widehat{n}_2(t)\\
  &= \frac{g_+^{1,2} + Zg_+^{3,4}}{2(1+Z)}\widehat{n_{20}} +\left(O(|\xi|) +iO(|\xi|)\right)\left(e^{- \kappa_1 |\xi|^2t +O(|\xi|^4t)} +e^{- \kappa_2 |\xi|^2t +O(|\xi|^4t)}\right)\widehat{n_{20}}
  \\
  & = \left(\frac1{1+Z}e^{- \kappa_1 |\xi|^2t}\cos\left(\sqrt{1+Z}t +\frac{\sigma_1^2}{2\sqrt{1+Z}}|\xi|^2t\right) +\frac Z{1+Z}e^{- \kappa_2 |\xi|^2t}\cos \left(\sigma_2|\xi|t\right)\right)\widehat{n_{20}}
  \\
  &\qquad +\left(O(|\xi|^4t) +O(|\xi|^3t) +O(|\xi|)+i\left(O(|\xi|^4t) +O(|\xi|^3t) +O(|\xi|)\right)\right)\left(e^{- \kappa_1 |\xi|^2t} +e^{- \kappa_2 |\xi|^2t}\right)\widehat{n_{20}}.
  \end{split}
 \end{equation}
 Then we have
 \begin{equation}\label{low-es-211}
  \begin{split}
  \|n_2^L\|_{L^2}^2 &= \|\widehat{n_2^L}\|_{L^2}^2
  \\
  &\geq \int_{|\xi|\le \eta}\left(\frac{e^{- \kappa_1 |\xi|^2t}}{1+Z}\cos\left(\sqrt{1+Z}t +\frac{\sigma_1^2}{2\sqrt{1+Z}}|\xi|^2t\right) +\frac {Ze^{- \kappa_2 |\xi|^2t}}{1+Z}\cos \left(\sigma_2|\xi|t\right)\right)^2|\widehat{n_{20}}|^2d\xi
  \\
  &\qquad +\int_{|\xi|\le \eta}\left(O(|\xi|^4t) +O(|\xi|^3t) +O(|\xi|)\right)^2\left(e^{- \kappa_1 |\xi|^2t} +e^{- \kappa_2 |\xi|^2t}\right)^2|\widehat{n_{20}}|^2d\xi
  \\
  & \geq \int_{|\xi|\le \eta}\frac{e^{- 2\kappa_1|\xi|^2t}}{(1+Z)^2}\cos^2\left(\sqrt{1+Z}t +\frac{\sigma_1^2}{2\sqrt{1+Z}}|\xi|^2t\right)|\widehat{n_{20}}|^2d\xi
  \\
  &\qquad +\int_{|\xi|\le \eta}\frac {Z^2e^{- 2\kappa_2 |\xi|^2t}}{(1+Z)^2}\cos^2 \left(\sigma_2|\xi|t\right)
  |\widehat{n_{20}}|^2d\xi
  \\
  &\qquad +\int_{|\xi|\le \eta}\frac{2Ze^{- (\kappa_1 +\kappa_2) |\xi|^2t}}{(1+Z)^2}\cos\left(\sqrt{1+Z}t +\frac{\sigma_1^2}{2\sqrt{1+Z}}|\xi|^2t\right)\cos \left(\sigma_2|\xi|t\right)|\widehat{n_{20}}|^2d\xi
  \\
  &\qquad-C(1+t)^{-\frac52}\|\widehat{n_{20}}\|_{L^\infty}^2
  \\
  &:= I_1 +I_2 +I_3 - C(1+t )^{-\frac52}\|n_{20}\|_{L^1}^2.
  \end{split}
 \end{equation}

 In spirit of \cite{LMZ} and \cite{WZZ2}, we can estimate the first two terms in the right hand side of \eqref{low-es-211} as
 \begin{equation}\label{I1}
  \begin{split}
  I_1 &= \int_{|\xi|\le \eta}\frac{e^{- 2\kappa_1|\xi|^2t}}{(1+Z)^2}\cos^2\left(\sqrt{1+Z}t +\frac{\sigma_1^2}{2\sqrt{1+Z}}|\xi|^2t\right)
  |\widehat{n_{20}}|^2d\xi
  \\
  &\geq C\delta_0^3t^{-\frac32}\int_0^{\eta\sqrt{t_0}} e^{- 2\kappa_1 r^2}\cos^2\left(\sqrt{1+Z}t +\frac{\sigma_1^2}{2\sqrt{1+Z}}r^2\right)dr
  \\
  &\geq C\delta_0^3t^{-\frac32}
  \end{split}
 \end{equation}
 and
 \begin{equation}\label{I2}
  \begin{split}
  I_2 &= \int_{|\xi|\le \eta}\frac {Z^2e^{- 2\kappa_2 |\xi|^2t}}{(1+Z)^2}\cos^2 \left(\sigma_2|\xi|t\right)^2
  |\widehat{n_{20}}|^2d\xi
  \\
  &= O(1)t^{-\frac32}\int_0^{\eta\sqrt{t}}e^{- 2\kappa_2 r^2}\left(\cos\left(2\sigma_2|\xi|t\right) +1\right)|\widehat{n_{20}}|^2dr
  \\
  &\geq C\delta_0^3t^{-\frac32}\int_0^{\eta\sqrt{t}}e^{- 2\kappa_2 r^2}dr
  -C\|\widehat{n_{20}}\|_{L^\infty}^2t^{-\frac32}\int_0^{\eta\sqrt{t}}e^{- 2\kappa_2 r^2}\cos\left(2\sigma_2|\xi|t\right)dr
  \\
  &\geq C\delta_0^3t^{-\frac32}- C\|n_{20}\|_{L^1}^2t^{-2}
  \end{split}
 \end{equation}
 for any $t\geq t_0$ with some sufficiently large time $t_0$ dependent of $\|v_{20}\|_{L^1}$.

 Moreover, we have that for $t\geq t_0$,
 \begin{equation}\label{I3}
  \begin{split}
  I_3&= \int_{|\xi|\le \eta}\frac{2Ze^{- (\kappa_1 +\kappa_2) |\xi|^2t}}{(1+Z)^2}\cos\left(\sqrt{1+Z}t +\frac{\sigma_1^2}{2\sqrt{1+Z}}|\xi|^2t\right)\cos \left(\sigma_2|\xi|t\right)|\widehat{n_{20}}|^2d\xi
  \\
  &= O(1)t^{-\frac32}\int_0^{\eta\sqrt{t}}e^{- (\kappa_1 +\kappa_2) r^2}\cos\left(\sqrt{1+Z}t +\frac{\sigma_1^2}{2\sqrt{1+Z}} r^2\right)
  \cos\left(\sigma_2 r\sqrt{t}\right)
  |\widehat{n_{20}}|^2dr
  \\
  &=O(1)t^{-\frac32}\int_0^{\eta\sqrt{t}}e^{- (\kappa_1 +\kappa_2) r^2}\cos\left(\sqrt{1+Z}t +\frac{\sigma_1^2}{2\sqrt{1+Z}} r^2\right)
  |\widehat{n_{20}}|^2\frac1{\sigma_2
  \sqrt{t}}d\sin\left(\sigma_2 r\sqrt{t}\right)
  \\
  &\geq -Ct^{-2}\|n_{20}\|_{L^1}^2,
  \end{split}
 \end{equation}
 thus plugging \eqref{I1}--\eqref{I3} into \eqref{low-es-211}, we can obtain that
 \begin{equation}\nonumber
  \| n_{2}^L\|_{L^2}
  \geq C\delta_0^\frac32(1+t)^{-\frac34} - C\|n_{20}\|_{L^1}(1+t)^{-1} - C\|n_{20}\|_{L^1}(1+t)^{-\frac54},
 \end{equation}
 which implies that \eqref{low-li-es1} holds for some appropriately positive constant $C_2$ and any large enough time $t$.
\end{proof}

\section{Proof of Upper Decay Estimates}

This section is devoted to prove the optimal decay rates of the solution stated in Theorem \ref{1mainth}. First by the general energy estimate method, we can derive the following energy inequality, whose proof can also be seen in \cite{WZZ}:
\begin{Proposition}\label{exist-u}
 Under the assumption \eqref{id-delta0} of Theorem \ref{1mainth}, the Cauchy problem \eqref{2eq} admits a unique globally classical solution $(\varrho_1, u_1, \varrho_2, u_2)$ such that for any $t \in [0,\infty)$,
 \begin{equation}\nonumber
  \begin{split}
  &\|(\varrho_1, u_1, \varrho_2, u_2, \nabla\phi)(t)\|_{H^l}^2 +\int_0^t\left(\|\nabla(\varrho_1, \varrho_2, \nabla\phi)(\tau)\|_{H^{l-1}}^2 +\|\nabla( u_1, u_2)(\tau)\|_{H^l}^2\right)d\tau
  \\
  &\le C\|(\varrho_{10}, u_{10}, \varrho_{20}, u_{20})\|_{H^l}^2
  \le CC_0.
  \end{split}
 \end{equation}
\end{Proposition}

Define the time--weighted energy functional
\begin{equation}\label{M1}
 \mathcal{M}(t) = \sup_{0\le \tau\le t}\sum_{k=0}^l\left((1+\tau)^{-\frac34-\frac k2}\|\nabla^k(\varrho_1, \varrho_2)\|_{L^2} +(1+\tau)^{\frac32(\frac1p-\frac12) + \frac k2} \|\nabla^k (m_1, m_2, \nabla\phi)(\tau)\|_{L^2}\right).
\end{equation}

We shall prove the following proposition to achieve the second part of Theorem \ref{1mainth}.
\begin{Proposition}\label{es-thm-M}
 Under the assumptions \eqref{id-delta0} and \eqref{k0} of Theorem \ref{1mainth}, it holds
 \begin{equation}\nonumber
  \mathcal{M}(t) \le C(C_0 +K_0).
 \end{equation}
\end{Proposition}

Next we divide the proof of Proposition \ref{es-thm-M} into the following steps: deriving the optimal decay rates on the low--frequent parts and the high--frequent parts of  solution and its highest order derivatives separately. To this end, we first need some tools to dealt with the integration on the time \cite{DUYZ1}.
\begin{lemma}\label{s1s2}
 Assume $s_1 >1$, $s_2\in [0, s_1]$, then we have
 \begin{equation}\nonumber
  \int_0^t(1+t-\tau)^{-s_1}(1+\tau)^{-s_2}d\tau
  \le C(s_1, s_2)(1+t)^{-s_2}.
 \end{equation}
\end{lemma}

Next we estimate the decay rate on the lower--frequent part of the solution as:
\begin{lemma}\label{es-fre-op}
 Assume that the assumptions of Proposition \ref{es-thm-M} are in force. Then it holds
 \begin{equation}\label{low-fre-varrho1}
  \|(\varrho_1^L, \varrho_2^L)(t)\|_{L^2}
  \le \left(CK_0 + C\mathcal{M}^2(t) \right)(1+t)^{-\frac34},
 \end{equation}
 and
 \begin{equation}\label{low-fre-m1}
  \|(m_1^L, m_2^L, (\nabla\phi)^L)(t)\|_{L^2}
  \le C\left(K_0+ \mathcal{M}^2(t) \right)(1+t)^{-\frac32\left(\frac1p - \frac12\right)}.
 \end{equation}
\end{lemma}
\begin{proof}
 To derive the decay on $\varrho_i^L$ and $m_i^L$ with $i = 1,2$, we need to estimate the nonlinear terms in \eqref{N1N2-m} as follows. By the definition \eqref{M1} of $\mathcal{M}(t)$, the h\"{o}lder inequality and the fact that $1\le p\le \frac32$, we have from \eqref{N1N2-m} that
 \begin{equation}\label{es-no-N}
  \begin{split}
  \|\mathcal{N}^L\|_{L^1}
  &\le C\|(\varrho_1, m_1, \varrho_2, m_2, \nabla\phi)\|_{L^2}\|(\varrho_1, \nabla m_1, \varrho_2, \nabla m_2)\|_{H^1}
  \\
  &\le C(1+t)^{-\frac32\left(\frac1p - \frac12\right)}\mathcal{M}(t)(1+t)^{-\frac34}\mathcal{M}(t)
  \\
  &\le C(1+t)^{-\frac32\left(\frac1p - \frac12\right)-\frac34}\mathcal{M}^2(t),
  \end{split}
 \end{equation}
 and
 \begin{equation}\label{es-no-F}
  \begin{split}
  \|(\mathbb{F}_1^L, \mathbb{F}_2^L)\|_{L^1}
  &\le C\|(\varrho_1, m_1, \varrho_2, m_2)\|_{L^2}\|(\varrho_1, m_1, \varrho_2, m_2)\|_{H^1}
  \\
  &\le C(1+t)^{-3\left(\frac1p - \frac12\right)}\mathcal{M}^2(t),
  \end{split}
 \end{equation}

 By \eqref{li-varrho-1}, \eqref{decay-S-varrho}, \eqref{decay-S-n} with $k=0$ and \eqref{es-no-N}--\eqref{es-no-F}, we have from \eqref{so-ex} that
 \begin{equation}\label{low-fre-varrho-de1}
  \begin{split}
  &\|(\varrho_1^L, \varrho_2^L)(t)\|_{L^2}
  \\
  &= \|(\widehat{\varrho_1^L}, \widehat{\varrho_2^L})(t)\|_{L^2}
  \\
  &\le C(1+t)^{-\frac34}\|U_0\|_{L^1} + \int_0^{\frac t2}\|((S^\varrho_1)^L, (S^\varrho_2)^L)\|_{L^2}d\tau + \int_{\frac t2}^t\|((S^\varrho_1)^L, (S^\varrho_2)^L)\|_{L^2}d\tau
  \\
  &\le C(1+t)^{-\frac34}\|U_0\|_{L^1} + C\int_0^{\frac t2}(1+t-\tau)^{-\frac54}\left(\|(\nabla\phi)^L(\tau)\|_{L^2}^2 +\|(\mathcal{N}^L, \mathbb{F}_1^L, \mathbb{F}_2^L)(\tau)\|_{L^1}\right) d\tau
  \\
  &\qquad +C\int_{\frac t2}^t(1+t-\tau)^{-\frac34}\|\mathcal{N}^L(\tau)\|_{L^1}d\tau
  \\
  &\le (1+t)^{-\frac34}\|U_0\|_{L^1}+ C\int_0^{\frac t2}(1+t -\tau)^{-\frac54}\left((1+\tau)^{-3\left(\frac1p - \frac12\right)}\mathcal{M}^2(\tau) +(1+ \tau)^{-\frac 32\left(\frac1p - \frac12\right) -\frac34}\mathcal{M}^2(\tau)\right)d\tau
  \\
  &\qquad+ C\int_{\frac t2}^t(1+t-\tau)^{-\frac34}(1+ \tau)^{-\frac 32\left(\frac1p - \frac12\right) -\frac34}\mathcal{M}^2(\tau)d\tau
  \\
  &\le C(1+t)^{-\frac34}\|U_0\|_{L^1}+ C\mathcal{M}^2(t)(1+t)^{-\frac54}\int_0^{\frac t2}(1+\tau)^{-3\left(\frac1p - \frac12\right)}d\tau
  \\
  &\qquad+ C\mathcal{M}^2(t)(1+ t)^{-\frac 32\left(\frac1p - \frac12\right) -\frac34}\int_{\frac t2}^t(1+t-\tau)^{-\frac34}d\tau
  \\
  &\le C(K_0 +\mathcal{M}^2(t))(1+t)^{-\frac34},
  \end{split}
 \end{equation}
 where the monotonicity of $\mathcal{M}(t)$ is used. And \eqref{low-fre-varrho-de1} yields \eqref{low-fre-varrho1}.

 Next by \eqref{4eq}, \eqref{li-v-1}, \eqref{li-omega-1}, \eqref{decay-S-n} with $k=0$ and \eqref{es-no-N}, we have from \eqref{so-ex} that
 \begin{equation}\nonumber
  \begin{split}
  &\|(m_1^L, m_2^L, (\nabla\phi)^L)(t)\|_{L^2}
  \\
  &= \|(n_1^L, n_2^L, M_1^L, M_2^L, (\nabla\phi)^L)(t)\|_{L^2}
  \\
  &\le C(1+t)^{-\frac34}\|U_0\|_{L^1} + C(1+t)^{-\frac32\left(\frac1p - \frac12\right)}\|\nabla\phi_0\|_{L^p} + C\int_0^t (1+t-\tau)^{-\frac34}\|\mathcal{N}^L(\tau)\|_{L^1}d\tau
  \\
  &\le C(1+t)^{-\frac32\left(\frac1p - \frac12\right)}K_0 + C\mathcal{M}^2(t)(1+t)^{-\frac34}\int_0^{\frac t2}(1+\tau)^{-\frac32\left(\frac1p - \frac12\right)-\frac34}d\tau
  \\
  &\qquad+ C\mathcal{M}^2(t)(1+ t)^{-\frac 32\left(\frac1p - \frac12\right) -\frac34}\int_{\frac t2}^t(1+t-\tau)^{-\frac34}d\tau
  \\
  &\le C(1+t)^{-\frac32\left(\frac1p - \frac12\right)}(K_0 + \mathcal{M}^2(t)),
  \end{split}
 \end{equation}
 which gets rise to \eqref{low-fre-m1}. Hence we finish the proof of Lemma \ref{es-fre-op}.
\end{proof}

\begin{lemma}\label{le-es-loworder1}
 Assume that the assumptions of Proposition \ref{es-thm-M} are in force. Then it holds
 \begin{equation}\nonumber
  \|\nabla^l(\varrho_1^L, \varrho_2^L)(t)\|_{L^2}
  \le C\left(K_0 + \mathcal{M}^2(t)\right)(1+t)^{-\frac34 - \frac l2}
 \end{equation}
 and
 \begin{equation}\nonumber
  \left\|\nabla^l(m_1^L, m_2^L, (\nabla\phi)^L)(t)\right\|_{L^2}
  \le C(K_0 +\mathcal{M}^2(t))(1+t)^{-\frac32\left(\frac1p -\frac12\right) - \frac l2}.
 \end{equation}
\end{lemma}
\begin{proof}
 Since by Lemma \ref{1interpolation}, Lemma \ref{es-product} and Lemma \ref{infty}, we can estimate the derivatives of the nonlinear term \eqref{N1N2-m} as
 \begin{equation}\label{es-no-N-high-L2}
  \begin{split}
  \|\nabla^l\mathcal{N}^L\|_{L^2}
  &\lesssim \|\nabla^l(\varrho_1\nabla\phi, \varrho_2\nabla\phi)\|_{L^2}
  \\
  &\qquad +\Bigg\|\nabla^l\left(\frac{m_1\otimes m_1}{\varrho_1}, \left(P_1(\rho_1) - P_1\left(\frac1 Z\right) - P'_1\left(\frac1 Z\right)\varrho_1\right)\mathbb{I}_3, \right.
  \\
  &\qquad\qquad\qquad\left. \frac{m_2\otimes m_2}{\varrho_2}, \left(P_2(\rho_2) - P_2(1) - P'_2(1)\varrho_2\right)\mathbb{I}_3\right)\Bigg\|_{L^2}
  \\
  &\qquad +\Bigg\|\nabla^{l-1}\left(\nabla\left(\frac{\varrho_1m_1}{\rho_1}\right), {\rm div}\left(\frac{\varrho_1m_1}{\rho_1}\right), \nabla\left(\frac{\varrho_2m_2}{\rho_2}\right), {\rm div}\left(\frac{\varrho_2m_2}{\rho_2}\right)\right)\Bigg\|_{L^2}
  \\
  &\lesssim \|(\varrho_1, m_1, \varrho_2, m_2, \nabla\phi)\|_{L^\infty}\|\nabla^l(\varrho_1, m_1, \varrho_2, m_2, \nabla\phi)\|_{L^2}
  \\
  &\lesssim \|(\varrho_1, m_1, \varrho_2, m_2, \nabla\phi)\|_{L^2}^\frac14\|\nabla^2(\varrho_1, m_1, \varrho_2, m_2, \nabla\phi)\|_{L^2}^\frac34\|\nabla^l(\varrho_1, m_1, \varrho_2, m_2, \nabla\phi)\|_{L^2}
  \\
  &\lesssim (1+t)^{-3\left(\frac1p - \frac12\right)-\frac34 -\frac l2}\mathcal{M}^2(t).
  \end{split}
 \end{equation}
 Here the fact that $\|\nabla f^L\|_{L^2}\lesssim \|f\|_{L^2}$ is used to dealt with the $l+1$--order derivatives. And similarly we have
 \begin{equation}\label{es-no-N-high-L1}
  \begin{split}
  \|\nabla^l\mathcal{N}^L\|_{L^1}
  &\lesssim \|\nabla^l(\varrho_1\nabla\phi, \varrho_2\nabla\phi)\|_{L^1}
  \\
  &\qquad +\Bigg\|\nabla^{l-1}{\rm div}\left(\frac{m_1\otimes m_1}{\varrho_1}, \left(P_1(\rho_1) - P_1\left(\frac1 Z\right) - P'_1\left(\frac1 Z\right)\varrho_1\right)\mathbb{I}_3, \right.
  \\
  &\qquad\qquad\qquad\qquad\left. \frac{m_2\otimes m_2}{\varrho_2}, \left(P_2(\rho_2) - P_2(1) - P'_2(1)\varrho_2\right)\mathbb{I}_3\right)\Bigg\|_{L^1}
  \\
  &\qquad +\Bigg\|\nabla^{l-2}{\rm div}\left(\nabla\left(\frac{\varrho_1m_1}{\rho_1}\right), {\rm div}\left(\frac{\varrho_1m_1}{\rho_1}\right), \nabla\left(\frac{\varrho_2m_2}{\rho_2}\right), {\rm div}\left(\frac{\varrho_2m_2}{\rho_2}\right)\right)\Bigg\|_{L^1}
  \\
  &\lesssim \|(\varrho_1, m_1, \varrho_2, m_2, \nabla\phi)\|_{L^2}\|\nabla^l(\varrho_1, m_1, \varrho_2, m_2, \nabla\phi)\|_{L^2}
  \\
  &\lesssim (1+t)^{-3\left(\frac1p - \frac12\right)-\frac l2}\mathcal{M}^2(t).
  \end{split}
 \end{equation}

 Thus by using \eqref{li-varrho-1}, \eqref{decay-S-varrho}, \eqref{decay-S-varrho2} with $k=l$, \eqref{es-no-N}, \eqref{es-no-F}, and \eqref{es-no-N-high-L2}, we have from \eqref{so-ex} that
 \begin{align}
  \notag&\|\nabla^l(\varrho_1^L, \varrho_2^L)(t)\|_{L^2}
  \\
  \notag&\le C(1+t)^{-\frac34 -\frac l2}\|U_0\|_{L^1} + \int_0^{\frac t2}\|\nabla^l((S^\varrho_1)^L, (S^\varrho_2)^L)\|_{L^2}d\tau + \int_{\frac t2}^t\|\nabla^l((S^\varrho_1)^L, (S^\varrho_2)^L)\|_{L^2}d\tau
  \\
  \notag&\le C(1+t)^{-\frac34 -\frac l2}\|U_0\|_{L^1} + C\int_0^{\frac t2}(1+t-\tau)^{-\frac54 -\frac l2}\left(\|(\nabla\phi)^L(\tau)\|_{L^2}^2 +\|(\mathcal{N}, \mathbb{F}_1, \mathbb{F}_2)(\tau)\|_{L^1}\right) d\tau
  \\
  \notag&\qquad +C\int_{\frac t2}^t
  (1+t-\tau)^{-\frac12}\|\nabla^l\mathcal{N}^L(\tau)\|_{L^2}d\tau
  \\
  \notag&\le (1+t)^{-\frac34 -\frac l2}\|U_0\|_{L^1}+ C\int_0^{\frac t2}(1+t -\tau)^{-\frac54 -\frac l2}(1+\tau)^{-3\left(\frac1p - \frac12\right)}\mathcal{M}^2(\tau)d\tau
  \\
  \notag&\qquad+ C\int_{\frac t2}^t(1+t-\tau)^{-\frac12}(1+ \tau)^{-3\left(\frac1p - \frac12\right) -\frac34 -\frac l2}\mathcal{M}^2(\tau)d\tau
  \\
  \notag&\le C(1+t)^{-\frac34 -\frac l2}\|U_0\|_{L^1}+ C\mathcal{M}^2(t)(1+t)^{-\frac54 -\frac l2}\int_0^{\frac t2}(1+\tau)^{-3\left(\frac1p - \frac12\right)}d\tau
  \\
  \notag&\qquad+ C\mathcal{M}^2(t)(1+ t)^{-3\left(\frac1p - \frac12\right) - \frac34 -\frac l2}\int_{\frac t2}^t(1+t-\tau)^{-\frac12}d\tau
  \\
  \notag&\le C\left(K_0 +\mathcal{M}^2(t)\right)(1+t)^{-\frac34 -\frac l2}.
 \end{align}

 Moreover, by using \eqref{4eq}, \eqref{li-v-1}, \eqref{li-omega-1}, \eqref{decay-S-n}, \eqref{decay-S-n2} with $k=l$, and \eqref{es-no-N}--\eqref{es-no-N-high-L1}, we have from \eqref{so-ex} that
 \begin{align*}
  &\|\nabla^l(m_1^L, m_2^L, (\nabla\phi)^L)(t)\|_{L^2}
  \\
  &= \|\nabla^l(n_1^L, n_2^L, M_1^L, M_2^L, (\nabla\phi)^L)(t)\|_{L^2}
  \\
  &\le C(1+t)^{-\frac34 -\frac l2}\|U_0\|_{L^1} + C(1+t)^{-\frac32\left(\frac 1p -\frac 12\right) -\frac l2}\|\nabla\phi_0\|_{L^p} +\int_0^\frac t2 (1+t-\tau)^{-\frac34 -\frac l2}\|\mathcal{N}^L(\tau)\|_{L^1}d\tau
  \\
  &\qquad +\int_\frac t2 ^t (1+t-\tau)^{-\frac34}\|\nabla^l\mathcal{N}^L(\tau)\|_{L^1}d\tau
  \\
  &\le CK_0(1+t)^{-\frac32\left(\frac 1p -\frac 12\right) -\frac l2} +C\mathcal{M}^2(t)(1+t)^{-\frac34 -\frac l2}\int_0^{\frac t2}(1+\tau)^{-\frac32\left(\frac1p - \frac12\right)-\frac34}d\tau
  \\
  &\qquad+ C\mathcal{M}^2(t)(1+ t)^{-3\left(\frac1p - \frac12\right) -\frac l2}\int_{\frac t2}^t(1+t-\tau)^{-\frac34}d\tau
  \\
  &\le C(1+t)^{-\frac32\left(\frac1p - \frac12\right) -\frac l2}(K_0 + \mathcal{M}^2(t)).
 \end{align*}

 Thus we finish the proof of Lemma \ref{le-es-loworder1}.
\end{proof}

Now we turn to estimate the decay rate on the higher--frequency of the highest--order derivative of the solution, and we state the result in the following:
\begin{lemma}\label{le-es-highorder1}
 Assume that the assumptions of Proposition \ref{es-thm-M} are in force. Then it holds
 \begin{equation}\label{high-fre-varrho-m1}
  \|\nabla^l(\varrho_1^H, m_1^H, \varrho_2^H, m_2^H, (\nabla\phi)^H)(t)\|_{L^2}
  \le C\left(C_0 + \mathcal{M}^2(t)\right)(1+t)^{-\frac34 - \frac l2}.
 \end{equation}
\end{lemma}
\begin{proof}
 Applying the operator $\mathfrak{F}^{-1}[(1-\varphi(\xi))\mathfrak{F}(\cdot)]$ to the system \eqref{2eq} gets rise to
 \begin{equation}\label{2eq-H}
  \left\{\begin{array}{lll}
  \partial_t \varrho_1^H + \frac1Z{\rm div}u_1^H = (N^\varrho_1)^H,
  \\
  \partial_t u_1^H + P'_1\left(\frac1Z\right)\nabla \varrho_1^H - Z\nabla \Delta^{-1} (Z\varrho_1^H - \varrho_2^H) -\mu_1Z \Delta u_1^H -\nu_1Z \nabla {\rm div} u_1^H = (N^u_1)^H,
  \\
  \partial_t \varrho_2^H + {\rm div}u_2^H =(N^\varrho_2)^H,
  \\
  \partial_t u_2^H + P'_2(1)\nabla \varrho_2^H + \nabla \Delta^{-1} (Z\varrho_1^H - \varrho_2^H) -\mu_2 \Delta u_2^H -\nu_2 \nabla {\rm div} u_2^H = (N^u_2)^H,
  \\
  (\varrho_1^H, u_1^H, \varrho_2^H, u_2^H)(x,0)= (\varrho_{10}^H, u_{10}^H, \varrho_{20}^H, u_{20}^H)(x)
  \end{array}\right.
 \end{equation}
 By taking $\langle \nabla^l\eqref{2eq-H}_1, \nabla^{l-1}u_1^H\rangle +\langle \nabla^l\varrho_1^H, \nabla^{l-1}\eqref{2eq-H}_2\rangle +\langle \nabla^l\eqref{2eq-H}_3, \nabla^{l-1}u_2^H\rangle +\langle \nabla^l\varrho_2^H, \nabla^{l-1}\eqref{2eq-H}_4\rangle $, we can obtain
 \begin{align}\label{high-es2}
 \begin{split}
  &\frac{d}{dt}\left(\langle \nabla^l\varrho_1^H, \nabla^{l-1}u_1^H\rangle +\langle \nabla^l\varrho_2^H, \nabla^{l-1}u_2^H\rangle\right) + P_1'\left(\frac1Z\right)\|\nabla^l\varrho_1^H\|_{L^2}^2+ P_2'(1)\|\nabla^l\varrho_2^H\|_{L^2}^2+ \|\nabla^l(\nabla\phi)^H\|_{L^2}^2\\
  &= \left\langle \nabla^l\left(- \frac1Z{\rm div}u_1^H +(N^\varrho_1)^H\right), \nabla^{l-1}u_1^H\right\rangle +\langle \nabla^l\left(-{\rm div}u_2^H + (N^\varrho_2)^H\right), \nabla^{l-1}u_2^H\rangle
 \\
  &\qquad +\langle \nabla^l\varrho_1^H, \nabla^{l-1}\left( \mu_1Z \Delta u_1^H +\nu_1Z \nabla{\rm div}u_1^H +(N^u_1)^H\right)\rangle
\\
  &\qquad +
  \langle \nabla^l\varrho_2^H, \nabla^{l-1}\left( \mu_2 \Delta u_2^H +\nu_2 \nabla{\rm div}u_2^H +(N^u_2)^H\right)\rangle
 \\
  &\le \left\langle \nabla^l\left(\frac1Z u_1^H +(\varrho_1u_1)^H\right), \nabla^{l}u_1^H\right\rangle +\langle \nabla^l\left(u_2^H +(\varrho_2u_2)^H\right), \nabla^{l}u_2^H\rangle
  \\
  &\qquad + C\big\|\nabla^l\left(\varrho_1^H, \varrho_2^H\right)\big\|_{L^2} \left(\big\|\nabla^{l+1}\left(u_1^H, u_2^H\right)\big\|_{L^2} +\big\|\nabla^{l-1}\left((N^u_1)^H, (N^u_2)^H\right)\big\|_{L^2}\right)
 \\
  &\le  \frac{P_1'\left(\frac1Z\right)}2\|\nabla^l\varrho_1^H\|_{L^2}^2+ \frac{P_2'(1)}2\|\nabla^l\varrho_2^H\|_{L^2}^2+ C\|\nabla^{l}\left(u_1^H, u_2^H, \nabla u_1^H, \nabla u_2^H\right)\|_{L^2}^2
 \\
  &\qquad +C\left( \big\|\nabla^l\left((\varrho_1u_1)^H, (\varrho_2u_2)^H\right)\big\|_{L^2}^2 +\big\|\nabla^{l-1}\left((N^u_1)^H, (N^u_2)^H\right)\big\|_{L^2}\right)
\\
  &\le  \frac{P_1'\left(\frac1Z\right)}2\|\nabla^l\varrho_1^H\|_{L^2}^2+ \frac{P_2'(1)}2\|\nabla^l\varrho_2^H\|_{L^2}^2 +C\|\nabla^{l+1}\left(u_1^H, u_2^H\right)\|_{L^2}^2
  \\
  &\qquad +C\left(\big\|\nabla^l\left(\varrho_1u_1, \varrho_2u_2\right)\big\|_{L^2}^2 +\big\|\nabla^{l-1}\left(N^u_1, N^u_2\right)\big\|_{L^2}^2\right),
 \end{split}
 \end{align}
 where the inequality $\|f^H\|_{L^2}\le \|f\|_{L^2}$ is used. Since by Lemma \ref{es-product} and Lemma \ref{infty}, we have
 \begin{equation}\label{es-no11}
  \big\|\nabla^l\left(\varrho_1u_1, \varrho_2u_2\right)\big\|_{L^2}
  \le C\|(\varrho_1, u_1, \varrho_2, u_2)\|_{L^\infty}\|\nabla^l(\varrho_1, u_1, \varrho_2, u_2)\|_{L^2},
 \end{equation}
 and
 \begin{equation}\label{es-no12}
  \begin{split}
  &\big\|\nabla^{l-1}\left(N^u_1, N^u_2\right)\big\|_{L^2}
  \\
  &\lesssim \Bigg\|\nabla^{l-1}\left(u_1\cdot \nabla u_1, \left( \frac {P'_1(\rho_1)} {\rho_1}-ZP'_1\left(\frac1Z\right) \right)\nabla \varrho_1, u_2\cdot \nabla u_2, \left( \frac {P'_2(\rho_2)} {\rho_2}-P'_2(1) \right)\nabla \varrho_2\right)\Bigg\|_{L^2}
  \\
  &\qquad +\Bigg\|\nabla^{l-1}\left(\frac {\varrho_1}{\rho_1}\Delta u_1, \frac {\varrho_1}{\rho_1} \nabla {\rm div} u_1, \frac {\varrho_2}{\rho_2}\Delta u_2, \frac {\varrho_2}{\rho_2} \nabla {\rm div} u_2\right)\Bigg\|_{L^2}
  \\
  &\lesssim\|(\varrho_1, u_1, \varrho_2, u_2)\|_{L^\infty}\|\nabla^l(\varrho_1, u_1, \varrho_2, u_2)\|_{L^2} +\|\nabla(\varrho_1, u_1, \varrho_2, u_2)\|_{L^3}\|\nabla^{l-1}(\varrho_1, u_1, \varrho_2, u_2)\|_{L^6}
  \\
  &\qquad +\|(\varrho_1, \varrho_2)\|_{L^\infty}\|\nabla^{l+1}(u_1, u_2)\|_{L^2} + \|\nabla^2(u_1, u_2)\|_{L^3}\|\nabla^{l-1}(\varrho_1, \varrho_2)\|_{L^6}
  \\
  &\lesssim \left(\|(\varrho_1, u_1, \varrho_2, u_2)\|_{L^\infty} +\|\nabla(\varrho_1, u_1, \varrho_2, u_2)\|_{1,3}\right)\|\nabla^l(\varrho_1, u_1, \varrho_2, u_2)\|_{L^2}
  \\
  &\qquad +\|(\varrho_1, \varrho_2)\|_{L^\infty}\left(\|\nabla^{l+1}(u_1^H, u_2^H)\|_{L^2} +\|\nabla^{l}(u_1, u_2)\|_{L^2}\right)
  \\
  &\lesssim \left(\|(\varrho_1, u_1, \varrho_2, u_2)\|_{L^\infty} +\|\nabla(\varrho_1, u_1, \varrho_2, u_2)\|_{1,3}\right)\|\nabla^l(\varrho_1, u_1, \varrho_2, u_2)\|_{L^2}
  \\
  &\qquad +\|(\varrho_1, \varrho_2)\|_{L^\infty}\|\nabla^{l+1}(u_1^H, u_2^H)\|_{L^2},
  \end{split}
 \end{equation}
 where $f = f^L +f^H$ and $\|\nabla f^L\|_{L^2}\lesssim \|f^L\|_{L^2}\lesssim \|f\|_{L^2}$ are used again. Thus plugging \eqref{es-no11} and \eqref{es-no12} into \eqref{high-es2}, we can arrive at
 \begin{equation}\label{high-es3}
  \begin{split}
  &\frac{d}{dt}\left(\langle \nabla^l\varrho_1^H, \nabla^{l-1}u_1^H\rangle +\langle \nabla^l\varrho_2^H, \nabla^{l-1}u_2^H\rangle\right) + \frac{P_1'\left(\frac1Z\right)}2\|\nabla^l\varrho_1^H\|_{L^2}^2+ \frac{P_2'(1)}2\|\nabla^l\varrho_2^H\|_{L^2}^2+ \|\nabla^l(\nabla\phi)^H\|_{L^2}^2
  \\
  &\le  C_3\|\nabla^{l+1}\left(u_1^H, u_2^H\right)\|_{L^2}^2 +C\left(\|(\varrho_1, u_1, \varrho_2, u_2)\|_{L^\infty} +\|\nabla(\varrho_1, u_1, \varrho_2, u_2)\|_{1,3}\right)^2\|\nabla^l(\varrho_1, u_1, \varrho_2, u_2)\|_{L^2}^2
  \end{split}
 \end{equation}
 with some positive constant $C_3$.

 Next by taking $ZP'_1\left(\frac1Z\right)\langle \nabla^l\eqref{2eq-H}_1^H, \nabla^l\varrho_1^H\rangle +\langle \nabla^l\eqref{2eq-H}_2^H, \nabla^lu_1^H\rangle +P'_2(1)\langle \nabla^l\eqref{2eq-H}_3^H, \nabla^l\varrho_1^H\rangle +\langle \nabla^l\eqref{2eq-H}_4^H, \nabla^lu_2^H\rangle$, we have
 \begin{equation}\label{high-es3-new1}
  \begin{split}
  &\frac{d}{dt}\left\{ \frac{ZP_1'\left(\frac1Z\right)}2\|\nabla^l\varrho_1^H\|_{L^2}^2 +\frac12\|\nabla^lu_1^H\|_{L^2}^2
  +\frac{P_2'(1)}2\|\nabla^l\varrho_2^H\|_{L^2}^2 +\frac12\|\nabla^lu_2^H \|_{L^2}^2+ \frac12\|\nabla^l\nabla\phi\|_{L^2}^2 \right\}
  \\
  &\quad+\mu_1\|\nabla^{l+1}u_1^H\|_{L^2}^2 +\nu_1\|\nabla^l{\rm div}u_1^H\|_{L^2}^2+\mu_2 \|\nabla^{l+1}u_2^H\|_{L^2}^2 +\nu_2\|\nabla^l{\rm div}u_2^H\|_{L^2}^2
  \\
  & = ZP_1'\left(\frac1Z\right)\langle \nabla^l (N^\varrho_1)^H, \nabla^l\varrho_1^H\rangle +\langle \nabla^l(N^u_1)^H, \nabla^lu_1^H\rangle +P_2'(1)\langle \nabla^l (N^\varrho_2)^H, \nabla^l\varrho_2^H\rangle +\langle \nabla^l(N^u_2)^H, \nabla^lu_2^H\rangle.
  \end{split}
 \end{equation}
 Here we only estimate the first two terms in the righthand side of \eqref{high-es3-new1} as follows. By taking full use of the properties of the low--frequency and high--frequency decomposition, Proposition \ref{exist-u}, the smallness of $C_0$, the Cauchy inequality and Lemma \ref{1commutator}, we can arrive at
 \begin{equation}\nonumber
  \begin{split}
  &\langle \nabla^l (N^\varrho_1)^H, \nabla^l\varrho_1^H\rangle
  \\
  &= -\left\langle \nabla^l {\rm div}\left(\left(\varrho_1^H +\varrho_1^L\right)u_1 -\left(\varrho_1u_1\right)^L\right), \nabla^l\varrho_1^H\right\rangle
  \\
  &= -\left\langle \nabla^l\left(\nabla \varrho^H \cdot u +\varrho^H {\rm div}u\right) + \nabla^l {\rm div}\left(\varrho_1^Lu_1 -\left(\varrho_1u_1\right)^L\right), \nabla^l\varrho_1^H\right\rangle
  \\
  & =\int_{\mathbb{R}^3}({\rm div}u_1)|\nabla^l\varrho_1^H|^2dx -\left\langle [\nabla^l, u_1]\cdot\nabla\varrho_1^H +\nabla^l\left(\varrho_1^H{\rm div}u_1\right) +\nabla^l{\rm div}\left(\varrho_1^Lu_1 -(\varrho_1u_1)^L\right), \nabla^l\varrho_1^H\right\rangle
  \\
  &\le \|\nabla u_1\|_{L^\infty}\|\nabla^l\varrho_1^H\|_{L^2}^2 +C\left(\|(\varrho_1, u_1)\|_{1, \infty}\|\nabla^l(\varrho_1, u_1)\|_{L^2} +\|\varrho_1\|_{L^\infty}\|\nabla^{l+1}u_1\|_{L^2}\right)
  \|\nabla^l\varrho_1^H\|_{L^2}
  \\
  &\le \|\nabla u_1\|_{H^2}\|\nabla^l\varrho_1^H\|_{L^2}^2 +C\left(\|(\varrho_1, u_1)\|_{1, \infty}\|\nabla^l(\varrho_1, u_1)\|_{L^2} +\|\varrho_1\|_{L^\infty}\|\nabla^{l+1}u_1^H\|_{L^2}\right)
  \|\nabla^l\varrho_1^H\|_{L^2}
  \\
  &\le \|(\varrho_1, u_1)\|_3\|\nabla^l\varrho_1^H\|_{L^2}^2 +C\|\varrho_1\|_{L^\infty}\|\nabla^{l+1}u_1^H\|_{L^2}^2 +C\|(\varrho_1, u_1)\|_{1, \infty}\|\nabla^l(\varrho_1, u_1)\|_{L^2}^2.
  \end{split}
 \end{equation}
 And by \eqref{es-no12}, we can get
 \begin{equation}\nonumber
  \begin{split}
   &\langle \nabla^l (N^u_1)^H, \nabla^lu_1^H\rangle
   \\
   &= -\langle \nabla^{l-1}(N^u_1)^H, \nabla^{l-1}\Delta u_1^H \rangle
   \\
   & \le \|\nabla^{l-1}(N^u_1)^H\|_{L^2}\|\nabla^{l+1}u_1^H\|_{L^2}
   \\
   &\le \frac{\mu_1}4\|\nabla^{l+1}u_1^H\|_{L^2} + C\|\nabla^{l-1}(N^u_1)^H\|_{L^2}^2
   \\
   &\le\left( \frac{\mu_1}4 +\|\varrho_1\|_{L^\infty}^2\right)\|\nabla^{l+1}u_1^H\|_{L^2}^2 +C\left(\|(\varrho_1, u_1)\|_{L^\infty} +\|\nabla(\varrho_1, u_1)\|_{1,3}\right)\|\nabla^l(\varrho_1, u_1)\|_{L^2}^2
  \end{split}
 \end{equation}
 Similarly we can get the estimates for $\langle \nabla^l (N^\varrho_2)^H, \nabla^l\varrho_2^H\rangle$ and $\langle \nabla^l (N^u_2)^H, \nabla^l u_2^H\rangle$. Thus combining these estimates with \eqref{high-es3-new1} gets rise to
 \begin{equation}\label{high-es3-new12}
  \begin{split}
  &\frac{d}{dt}\left\{ \frac{ZP_1'\left(\frac1Z\right)}2\|\nabla^l\varrho_1^H\|_{L^2}^2 +\frac12\|\nabla^lu_1^H\|_{L^2}^2
  +\frac{P_2'(1)}2\|\nabla^l\varrho_2^H\|_{L^2}^2 +\frac12\|\nabla^lu_2^H \|_{L^2}^2+ \frac12\|\nabla^l\nabla\phi\|_{L^2}^2 \right\}
  \\
  &\quad+\frac{\mu_1}2\|\nabla^{l+1}u_1^H\|_{L^2}^2 +\frac{\mu_2}2 \|\nabla^{l+1}u_2^H\|_{L^2}^2\\
  &\le C\|(\varrho_1, u_1, \varrho_2, u_2)\|_3\|\nabla^l(\varrho_1^H, \varrho_2^H)\|_{L^2}^2\\
  &\qquad+C\left(\|(\varrho_1, u_1, \varrho_2, u_2)\|_{1, \infty} +\|\nabla(\varrho_1, u_1, \varrho_2, u_2)\|_{1, 3}\right)\|\nabla^l(\varrho_1, u_1, \varrho_2, u_2)\|_{L^2}^2.
  \end{split}
 \end{equation}
 Now multiplying \eqref{high-es3} with some positive number $C_4=\min\left\{\frac{\mu_1}{4C_3}, \frac{\mu_2}{4C_3}\right\}$, and plus \eqref{high-es3-new12} leads to
 \begin{equation}\label{high-es3-new13}
  \begin{split}
  &\frac{d}{dt}\left\{C_4\langle\nabla^l\varrho_1^H, \nabla^{l-1}u_1^H\rangle + C_4\langle\nabla^l\varrho_2^H, \nabla^{l-1}u_2^H\rangle + \frac{ZP_1'\left(\frac1Z\right)}2\|\nabla^l\varrho_1^H\|_{L^2}^2 +\frac12\|\nabla^lu_1^H\|_{L^2}^2\right.
  \\
  &\qquad\left.+\frac{P_2'(1)}2\|\nabla^l\varrho_2^H\|_{L^2}^2 +\frac12\|\nabla^lu_2^H \|_{L^2}^2+ \frac12\|\nabla^l(\nabla\phi)^H\|_{L^2}^2 \right\}+ \frac{C_4P_1'\left(\frac1Z\right)}4\|\nabla^l\varrho_1^H\|_{L^2}^2
  \\
  &\quad+ \frac{C_4P_2'(1)}4\|\nabla^l\varrho_2^H\|_{L^2}^2+ C_4\|\nabla^l(\nabla\phi)^H\|_{L^2}^2 +\frac{\mu_1}4\|\nabla^{l+1}u_1^H\|_{L^2}^2 +\frac{\mu_2}4 \|\nabla^{l+1}u_2^H\|_{L^2}^2
  \\
  &\le C\left(\|(\varrho_1, u_1, \varrho_2, u_2)\|_{1, \infty} +\|\nabla(\varrho_1, u_1, \varrho_2, u_2)\|_{1, 3}\right)\|\nabla^l(\varrho_1, u_1, \varrho_2, u_2)\|_{L^2}^2,
  \end{split}
 \end{equation}

 Define
 \begin{equation}\nonumber
  \begin{split}
  \mathcal{L}(t)
  =& C_4\langle\nabla^l\varrho_1^H, \nabla^{l-1}u_1^H\rangle + C_4\langle\nabla^l\varrho_2^H, \nabla^{l-1}u_2^H\rangle + \frac{ZP_1'\left(\frac1Z\right)}2\|\nabla^l\varrho_1^H\|_{L^2}^2 +\frac12\|\nabla^lu_1^H\|_{L^2}^2
  \\
  &\quad+\frac{P_2'(1)}2\|\nabla^l\varrho_2^H\|_{L^2}^2 +\frac12\|\nabla^lu_2^H \|_{L^2}^2+ \frac12\|\nabla^l(\nabla\phi)^H\|_{L^2}^2.
  \end{split}
 \end{equation}
 By the Cauchy inequality and the fact that $\|f^H\|_{L^2}\le \|\nabla f^H\|_{L^2}$, we can get the following equivalent relationship
 \begin{equation}\nonumber
  \mathcal{L}(t)
  \approx \|\nabla^l(\varrho_1^H, u_1^H, \varrho_2^H, u_2^H, (\nabla\phi)^H)\|_{L^2}^2.
 \end{equation}
 Hence by using $\|f^H\|_{L^2}\le \|\nabla f^H\|_{L^2}$ again we have from \eqref{high-es3-new13} that for some positive constant $C_5$,
 \begin{equation}\nonumber
  \frac{d}{dt}\mathcal{L}(t) +C_5\mathcal{L}(t)
  \le C\left(\|(\varrho_1, u_1, \varrho_2, u_2)\|_{1, \infty} +\|\nabla(\varrho_1, u_1, \varrho_2, u_2)\|_{1, 3}\right)\|\nabla^l(\varrho_1, u_1, \varrho_2, u_2)\|_{L^2}^2.
 \end{equation}
 Then by the Gronwall inequality and Lemma \ref{s1s2}, we can arrive at
 \begin{equation}\nonumber
  \begin{split}
  L(t)
  &\le e^{-C_5t}\mathcal{L}(0) + \int_0^t e^{-C_5(t-\tau)}C\left(\|(\varrho_1, u_1, \varrho_2, u_2)(\tau)\|_{1, \infty} +\|\nabla(\varrho_1, u_1, \varrho_2, u_2)(\tau)\|_{1, 3}\right)
  \\
  &\qquad\qquad\qquad\qquad\times\|\nabla^l(\varrho_1, u_1, \varrho_2, u_2)(\tau)\|_{L^2}^2d\tau
  \\
  &\le e^{-C_5t}\mathcal{L}(0) + C\int_0^t e^{-C_5(t-\tau)}\|(\varrho_1, u_1, \varrho_2, u_2)(\tau)\|_1^\frac14\|\nabla^2(\varrho_1, u_1, \varrho_2, u_2)(\tau)\|_1^\frac34
  \\
  &\qquad\qquad\qquad\qquad\times\|\nabla^l(\varrho_1, u_1, \varrho_2, u_2)(\tau)\|_{L^2}^2d\tau
  \\
  &\le e^{-C_5t}C_0^2 + C\int_0^t e^{-C_5(t-\tau)}(1+\tau)^{-\frac92\left(\frac1p -\frac12\right) -\frac34 -l}\mathcal{M}^4(\tau)d\tau
  \\
  &\le e^{-C_5t}C_0^2 + C\int_0^t (1+t-\tau)^{-\frac92\left(\frac1p -\frac12\right) -\frac34 -l}(1+\tau)^{-\frac92\left(\frac1p -\frac12\right) -\frac34 -l}\mathcal{M}^4(\tau)d\tau
  \\
  &\le (1+t)^{-\frac92\left(\frac1p -\frac12\right) -\frac34 -l}(C_0^2 +C\mathcal{M}^4(t))
  \\
  &\le C(C_0^2 +\mathcal{M}^4(t))(1+t)^{-\frac32 -l},
  \end{split}
 \end{equation}
 which together with the fact that $\|\nabla^l(m_1^H, m_2^H)\|_{L^2}\lesssim \|(\varrho_1^H, u_1^H, \varrho_2^H, u_2^H)\|_{L^\infty}\|\nabla^l(\varrho_1^H, u_1^H, \varrho_2^H, u_2^H)\|_{L^2}$ yields \eqref{high-fre-varrho-m1}.
\end{proof}

$Proof\ of\ Proposition\ \ref{es-thm-M}$.

 Combining with Lemma \ref{es-fre-op}, Lemma \ref{le-es-loworder1} and Lemma \ref{le-es-highorder1}, the following estimates can be obtained:
 \begin{equation}\label{low-varrho1}
  \begin{split}
  \|(\varrho_1, \varrho_2)\|_{L^2}
  &\le \|(\varrho_1^H, \varrho_2^H)\|_{L^2} +\|(\varrho_1^L, \varrho_2^L)\|_{L^2}
  \\
  &\le \|\nabla^l(\varrho_1^H, \varrho_2^H)\|_{L^2} +\|(\varrho_1^L, \varrho_2^L)\|_{L^2}
  \\
  &\le C(C_0 +K_0 +\mathcal{M}^2(t))(1+t)^{-\frac34},
  \end{split}
 \end{equation}

 \begin{equation}\label{low-m1}
  \begin{split}
  \|(m_1, m_2, \nabla\phi)\|_{L^2}
  &\le \|(m_1^H, m_2^H, (\nabla\phi)^H)\|_{L^2} +\|(m_1^L, m_2^L, (\nabla\phi)^L)\|_{L^2}
  \\
  &\le \|\nabla^l(m_1^H, m_2^H, (\nabla\phi)^H)\|_{L^2} +\|(m_1^L, m_2^L, (\nabla\phi)^L)\|_{L^2}
  \\
  &\le C(C_0 +K_0 +\mathcal{M}^2(t))(1+t)^{-\frac32\left(\frac1p -\frac12\right)},
  \end{split}
 \end{equation}
 \begin{equation}\label{high-varrho1}
  \begin{split}
  \|\nabla^l(\varrho_1, \varrho_2)\|_{L^2}
  &\le \|\nabla^l(m_1^H, m_2^H, (\nabla\phi)^H)\|_{L^2} +\|\nabla^l(m_1^L, m_2^L, (\nabla\phi)^L)\|_{L^2}
  \\
  &\le C(C_0 +K_0 +\mathcal{M}^2(t))(1+t)^{-\frac34 -\frac l2},
  \end{split}
 \end{equation}
 and
 \begin{equation}\label{high-m1}
  \|\nabla^l(m_1, m_2, \nabla\phi)\|_{L^2}
  \le \|\nabla^l(\varrho_1^H, \varrho_2^H)\|_{L^2} +\|\nabla^l(\varrho_1^L, \varrho_2^L)\|_{L^2}
  \le C(C_0 +K_0 +\mathcal{M}^2(t))(1+t)^{-\frac32\left(\frac1p -\frac12\right) -\frac l2}.
 \end{equation}

 Finally, by the definition \eqref{M1} of $\mathcal{M}(t)$, and using \eqref{low-varrho1}--\eqref{high-m1} and the Sobolev interpolation inequality, we can get
 \begin{equation}\nonumber
  \mathcal{M}(t) \le C(C_0 +K_0) + C\mathcal{M}^2(t),
 \end{equation}
 which together with the smallness of $C_0$ and $K_0$ implies that $\mathcal{M}(t) \le C(C_0 +K_0)$. This complete the proof of Proposition \ref{es-thm-M}.

\section{Proof of Lower Decay Estimates}

In this section, we prove the lower decay estimates of the solution and its derivatives to the system \eqref{1eq}.
\begin{Proposition}\label{es-thm-lower}
 Under the assumptions \eqref{id-delta0}, \eqref{k0} for $p=1$ and \eqref{low-as} of Theorem \ref{exist-u}, it holds that for $0\le k\le l$ and some positive constant $C_6$,
 \begin{equation}\nonumber
  \min\left\{\|\nabla^k\varrho_1\|_{L^2}, \|\nabla^km_1\|_{L^2}, \|\nabla^k\varrho_2\|_{L^2}, \|\nabla^km_2\|_{L^2}, \|\nabla^k\nabla\phi\|_{L^2}\right\}
  \geq C_6\delta_0^\frac32(1+t)^{-\frac34 - \frac k2}.
 \end{equation}
\end{Proposition}

First by using \eqref{4eq}, the lower decay estimate \eqref{low-li-es1} for the linear system and \eqref{es-no-N}, the decay estimate \eqref{decay-S-n2} with $k=0$ on the nonlinear term, Proposition \ref{es-thm-M} and Lemma \ref{s1s2}, we have from \eqref{so-ex} that
\begin{align}
 \notag&\min\left\{\|\varrho_1\|_{L^2}, \|m_1\|_{L^2}, \|\varrho_2\|_{L^2}, \|m_2\|_{L^2}, \|\nabla\phi\|_{L^2}\right\}
 \\
 \notag&\geq \min\left\{\|\varrho_1^L\|_{L^2}, \|m_1^L\|_{L^2}, \|\varrho_2^L\|_{L^2}, \|m_2^L\|_{L^2}, \|(\nabla\phi)^L\|_{L^2}\right\}
 \\
 \notag&\geq C\delta_0^\frac32(1+t)^{-\frac34} -C\left|\int_0^t(1+t-\tau)^{-\frac34}\|\mathcal{N}^L(\tau)\|_{L^1}
 d\tau\right|
 \\
 \notag&\geq C\delta_0^\frac32(1+t)^{-\frac34} -C\int_0^t(1+t-\tau)^{-\frac34}\|(\varrho_1, m_1, \varrho_2, m_2, \nabla\phi)\|_{L^2}\|(\varrho_1, \nabla m_1, \varrho_2, \nabla m_2)\|_{H^1}d\tau
 \\
 \notag&\geq C\delta_0^\frac32(1+t)^{-\frac34} -CC_0^\frac13\int_0^t(1+t-\tau)^{-\frac34}\|(\varrho_1, m_1, \varrho_2, m_2, \nabla\phi)\|_{L^2}^{\frac23}\|(\varrho_1, \nabla m_1, \varrho_2, \nabla m_2)\|_{H^1}d\tau
 \\
 \notag&\geq C\delta_0^\frac32(1+t)^{-\frac34} -C C_0^\frac13 \int_0^t(1+t-\tau)^{-\frac34}(1+\tau)^{-\frac54}
 \mathcal{M}^\frac53(\tau)d\tau
 \\
 \notag&\geq (C\delta_0^\frac32 - CC_0^\frac13(C_0 +K_0)^\frac53)(1+t)^{-\frac34},
\end{align}
this combining with the fact that $C_0, K_0 \le \delta_0$ is small, implies that
\begin{equation}\label{low-nonli-es1}
 \begin{split}
 \min\left\{\|\varrho_1\|_{L^2}, \|m_1\|_{L^2}, \|\varrho_2\|_{L^2}, \|m_2\|_{L^2}, \|\nabla\phi\|_{L^2}\right\}
 \geq C_7\delta_0^\frac32(1+t)^{-\frac34}
 \end{split}
\end{equation}
for some positive constant $C_7$.

As in the proof of \eqref{low-nonli-es1}, by using \eqref{decay-S-n} with $k=1$ and \eqref{one-low-li-es1}, we have from \eqref{so-ex} that
\begin{equation}\nonumber
 \begin{split}
 & \min\left\{\|\nabla\varrho_{1}\|_{L^2}, \|\nabla m_{1}\|_{L^2}, \|\nabla\varrho_{2}\|_{L^2}, \|\nabla m_{2}\|_{L^2}, \|\nabla(\nabla\phi)\|_{L^2}\right\}
 \\
 &\geq \min\left\{\|\nabla\varrho_{1}^L\|_{L^2}, \|\nabla m_{1}^L\|_{L^2}, \|\nabla\varrho_{2}^L\|_{L^2}, \|\nabla m_{2}^L\|_{L^2}, \|\nabla(\nabla\phi)^L\|_{L^2}\right\}
 \\
 &\geq C\delta_0^\frac32(1+t)^{-\frac54} -C\left|\int_0^t(1+t-\tau)^{-\frac54}
 \|\mathcal{N}(\tau)\|_{L^1}d\tau\right|
 \\
 &\geq C\delta_0^\frac32(1+t)^{-\frac54} -CC_0^\frac13\mathcal{M}^\frac53\int_0^t(1+t-\tau)^{-\frac54}(1+\tau)^{-\frac54}d\tau
 \\
 &\geq (C\delta_0^\frac32 -CC_0^\frac13(C_0 +K_0)^\frac53)(1+t)^{-\frac54},
 \end{split}
\end{equation}
which also yields that
\begin{equation}\label{one-low-nonli-es1}
 \begin{split}
 \|\nabla(\varrho_1, u_1, \varrho_2, u_2)\|_{L^2}
 \geq C_8\delta_0^\frac32(1+t)^{-\frac54}
 \end{split}
\end{equation}
for some positive constant $C_8$.

Finally, by using \eqref{low-nonli-es1}, \eqref{one-low-nonli-es1} and the Sobolev interpolation inequality, we can deduce that for $2\le k\le l$,
\begin{equation}\nonumber
 \|\nabla^k(\varrho_1, u_1, \varrho_2, u_2)\|_{L^2}
 \geq C\|\nabla(\varrho_1, u_1, \varrho_2, u_2)\|_{L^2}^k\|(\varrho_1, u_1, \varrho_2, u_2)\|_{L^2}^{-(k-1)}
 \geq C_9\delta_0^\frac32(1+t)^{-\frac34 - \frac k2}
\end{equation}
for some positive constant $C_9$. This finishes the proof of Proposition \ref{es-thm-lower}.

\section{Analytic tools}

We will extensively use the Sobolev interpolation of the Gagliardo--Nirenberg inequality; the proof can be seen in \cite{N3}.
\begin{lemma}\label{1interpolation}
 Let $0\le i, j\le k$, then we have
 \begin{equation}\nonumber
  \|\nabla^i f\|_{L^p}\lesssim \|\nabla^jf\|_{L^q}^{1-a}\| \nabla^k f\|_{L^r}^a
 \end{equation}
 where $a$ belongs to $\left[\frac ik, 1\right]$ and satisfies
 \begin{equation}\nonumber
  \frac{i}{3}-\frac{1}{p}= \left(\frac{j}{3}-\frac{1}{q}\right)(1-a)+ \left(\frac{k}{3}-\frac{1}{r}\right)a.
 \end{equation}

 Especially, while $p=q=r=2$, we have
 \begin{equation}\nonumber
  \|\nabla^if\|_{L^2}\lesssim \|\nabla^jf\|_{L^2}^\frac{k-i}{k-j}
  \|\nabla^kf\|_{L^2}^\frac{i-j}{k-j}.
 \end{equation}
\end{lemma}

To estimate the product of two functions, we shall record the following estimate, cf. \cite{J}:
\begin{lemma}\label{es-product}
 It holds that for $k\geq0$,
 \begin{equation}\nonumber
  \|\nabla ^k(gh)\|_{L^{p_0}} \lesssim \|g\|_{L^{p_1}}\|\nabla^kh\|_{L^{p_2}} +\|\nabla^kg\|_{L^{p_3}}\|h\|_{L^{p_4}}.
 \end{equation}
 Here $p_0, p_2, p_3\in (1, \infty)$ and
 \begin{equation}\nonumber
  \frac1{p_0} = \frac1{p_1} +\frac1{p_2} = \frac1{p_3} +\frac1{p_4}.
 \end{equation}
\end{lemma}

Thus we can easily deduce from Lemma \ref{es-product} the following commutator estimate:
\begin{lemma}\label{1commutator}
 Let $f$ and $g$ be smooth functions belonging to $H^k\cap L^\infty$ for any integer $k\ge1$ and  define the commutator
 \begin{equation}\nonumber
  [\nabla ^k,f]g=\nabla ^k(fg)-f\nabla ^kg.
 \end{equation}
 Then we have
 \begin{equation}\nonumber
  \|[\nabla ^k,f]g\|_{L^{p_0}} \lesssim \|\nabla  f\|_{L^{p_1}}\|\nabla ^{k-1}g\|_{L^{p_2}}+\|\nabla ^k f\|_{L^{p_3}}\| g\|_{L^{p_4}}.
 \end{equation}
 Here $p_i(i=0,1,2,3,4)$ are defined in Lemma \ref{es-product}.
\end{lemma}

Next, to estimate the $L^2$--norm of the spatial derivatives of some smooth function $F(f)$, we shall introduce some estimates which follow from Lemma \ref{1interpolation} and Lemma \ref{es-product}:
\begin{lemma}\label{infty}
 Let $F(f)$ be a smooth function of $f$ with bounded derivatives of any order and $f$ belong to $H^k$ for any integer $k\ge3$ , then we have
 \begin{equation}\nonumber
  \|\nabla ^k(F(f))\|_{L^2} \lesssim \sup_{0\le i\le k}\|F^{(i)}(f)\|_{L^\infty}
  \left(\sum_{j=2}^k\|f\|_{L^2}^{j-1-\frac{3(j-1)}{2k}}\|\nabla ^kf\|_{L^2}^{1+\frac{3(j-1)}{2k}}+ \|\nabla ^kf\|_{L^2}\right).
 \end{equation}

 Moreover, if $f$ has the lower and upper bounds, and $\|f\|_k\le 1$, we have
 \begin{equation}\nonumber
  \|\nabla ^k(F(f))\|_{L^2} \lesssim \|\nabla ^kf\|_{L^2}.
 \end{equation}
\end{lemma}

\bigskip

\subsection*{Acknowledgments}
Qing Chen's research is supported in part by Natural Science Foundation of Fujian Province (No. 2018J01430). Guochun Wu's research was in part supported by National Natural Science Foundation of China (No. 11701193, 11671086),  Natural Science Foundation of Fujian Province (No. 2018J05005, 2017J01562), Program for Innovative Research Team in Science and Technology in Fujian Province University  Quanzhou High--Level Talents Support Plan (No. 2017ZT012). Yinghui Zhang' research is partially supported by Guangxi Natural Science Foundation (No. 2019JJG110003, 2019AC20214), and National Natural Science Foundation of China (No. 11771150, 11571280, 11301172 and 11226170.)

\end{document}